\documentclass[11pt]{article}

\usepackage[english]{babel}
\usepackage[T1]{fontenc}
\usepackage{amsfonts}
\usepackage{amsmath}
\usepackage{amssymb}
\usepackage{amsthm}
\usepackage{mathtools}
\usepackage{stmaryrd}
\usepackage{dsfont}
\usepackage{graphicx}
\usepackage{hyperref}
\usepackage{titlesec}
\usepackage{mathabx}
\usepackage{enumitem,tocloft}
\usepackage[top=2.5cm, bottom=2.5cm, left=3cm, right=3cm]{geometry}
\usepackage{xcolor}
\usepackage{fourier-orns}
\usepackage[utf8]{inputenc}
\usepackage{bm}
\usepackage{framed}
\usepackage{supertabular}
\hypersetup{linktocpage,breaklinks=true}
\usepackage{mathrsfs}
\usepackage{csquotes}

\usepackage[backend=bibtex, style=alphabetic, maxbibnames=6, sorting=nyt]{biblatex}

\newcommand{\ld}{\mathrm{L}}
\newcommand{\R}{\mathbb{R}}
\newcommand{\Z}{\mathbb{Z}}
\newcommand{\N}{\mathbb{N}}
\newcommand{\Om}{\Omega}
\newcommand{\G}{\Gamma}
\newcommand{\La}{\Lambda}
\newcommand{\g}{\gamma}
\newcommand{\la}{\lambda}
\newcommand{\supp}{\mathrm{supp}}

\newcommand{\act}{\curvearrowright}

\newcommand{\gint}{\nabla^{\mathrm{int}}}
\newcommand{\gsup}{\nabla^{\sup}}
\newcommand\restr[2]{{
  \left.\kern-\nulldelimiterspace 
  #1 
  \vphantom{|} 
  \right|_{#2} 
  }}

\newtheorem{theorem}{Theorem}[section]
\newtheorem*{theorem*}{Theorem}

\newtheorem*{corollary*}{Corollary}

\newtheorem{proposition}[theorem]{Proposition}
\newtheorem{lemma}[theorem]{Lemma}

\newtheorem*{claim*}{Claim}

\newtheorem{theoremletter}{Theorem}

\theoremstyle{definition}
\newtheorem{definition}[theorem]{Definition}
\newtheorem{remark}[theorem]{Remark}

\title{Obstructions for quantitative measure equivalence between locally compact groups}
\author{Corentin Correia and Juan Paucar}
\date{\today}

\begin{document}

\maketitle

\begin{abstract}
    Given a measure equivalence coupling between two finitely generated groups, Delabie, Koivisto, Le Maître and Tessera have found explicit upper bounds on how integrable the associated cocycles can be. We extend these results to the broader framework of unimodular compactly generated locally compact groups. We also generalize a result by the first-named author, showing that the integrability threshold described in these statements cannot be achieved.
\end{abstract}

{
		\small	
		\noindent\textbf{{Keywords:}} Quantitative measure equivalence, locally compact groups, isoperimetric profile, volume growth.
	}
	
	\smallskip
	
	{
		\small	
		\noindent\textbf{{MSC-classification:}} Primary 37A20; Secondary 22D05, 22F10, 22D40, 20F69, 20F65.
	
}

\setcounter{tocdepth}{1}
\tableofcontents

\section{Introduction}

Gromov introduced measure equivalence for countable groups as a measurable analogue of quasi-isometry.

\begin{definition}
    Two countable groups $\G$ and $\La$ are \textit{measure equivalent} if there exists a standard measured space $(\Omega,\mu)$ equipped with commuting measure-preserving $\G$- and $\La$-actions such that
    \begin{enumerate}
        \item both actions are free;
        \item the $\G$- and $\La$-actions admit Borel fundamental domains $X_{\G}$ and $X_{\La}$ respectively;
        \item $X_{\G}$ and $X_{\La}$ have finite measure.
    \end{enumerate}
    The quadruple $(\Omega,X_{\G},X_{\La},\mu)$ is called a \textit{measure equivalence coupling} between $\G$ and $\La$. It is straightforward to prove that $\mu(X_{\G})$ and $\mu(X_{\La})$ are positive, and the restrictions of $\mu$ to the fundamental domains $X_{\G}$ and $X_{\La}$ are respectively denoted by $\nu_{\G}$ and $\nu_{\La}$.
\end{definition}

The most natural example to keep in mind is that of lattices in the same locally compact group. Another instance of measure equivalence comes from ergodic theory.

\begin{definition}[Dye~\cite{dyeGroupsMeasurePreserving1959}]
    Two countable groups $\G$ and $\La$ are \textit{orbit equivalent} if there exist free probability measure-preserving $\G$- and $\La$-actions on a standard probability space $(X,\mu)$ such that for almost every $x\in X$, $\G\cdot x=\La\cdot x$.
\end{definition}

Two groups are orbit equivalent if and only if they are measure equivalent with common fundamental domains (see~\cite{furmanSurveyMeasuredGroup2011} for a proof).\par
This paper deals with measure equivalence in the setting of locally compact groups, a precise definition in this context will be given in Section~\ref{sec:Prel:ME}. Measure and orbit equivalences have already been studied for such groups, see e.g.~\cite{baderIntegrableMeasureEquivalence2013,carderiOrbitFullGroups2017,koivistoMeasureEquivalenceNonunimodular2021,koivistoMeasureEquivalenceCoarse2021,paucarIsoperimetricProfilesRegular2024,delabie$mathrmL^p$MeasureEquivalence2025}, and it is worth mentioning that measure equivalence and orbit equivalence are the same notions among non-discrete, locally compact, second countable groups~\cite[Theorem~A]{koivistoMeasureEquivalenceCoarse2021}. Before stating our main results, we briefly review the discrete case\par
Many results have been obtained for countable non-amenable groups, most of which describe rigidity phenomena. We refer the reader to the works of
\begin{itemize}
    \item Furman on lattices in higher rank semi-simple Lie groups~\cite{furmanGromovsMeasureEquivalence1999};
    \item Kida on mapping class groups~\cite{kidaMappingClassGroup2008,kidaMeasureEquivalenceRigidity2010};
    \item Guirardel and Horbez on $\mathrm{Out}(F_N)$~\cite{guirardelMeasureEquivalenceRigidity2021};
    \item Horbez and Huang on right-angled Artin groups~\cite{horbezMeasureEquivalenceClassification2022,horbezMeasureEquivalenceClassification2024};
    \item Escalier and Horbez on graph products~\cite{escalierGraphProductsMeasure2024}.
\end{itemize}

However, Ornstein and Weiss proved that any two free probability measure-preserving actions of infinite amenable groups are orbit equivalent~\cite{ornsteinErgodicTheoryAmenable1980}. In particular, measure equivalence is trivial for such groups. In the non-discrete setting, Koivisto, Kyed and Raum proved that the class of amenable, locally compact, second countable consists in three measure equivalence classes: compact groups, non-compact unimodular amenable groups and non-unimodular
amenable groups~\cite[Theorem~B]{koivistoMeasureEquivalenceCoarse2021}. To counter this flexibility phenomenon, we impose restrictions on the functions arising from measure equivalence couplings, namely the \textit{cocycles}.

\begin{definition}
Let $(\Omega,X_{\G},X_{\La},\mu)$ be a measure equivalence coupling between discrete groups $\G$ and $\La$. For every $\g\in\G$ and $\la\in\La$, and for almost every $x_{\G}\in X_{\G}$ and $x_{\La}\in X_{\La}$, there exist unique $\alpha(\g,x_{\La})\in\La$ and $\beta(\la,x_{\G})\in\G$ such that
$$\alpha(\g,x_{\La})\ast \g\ast x_{\La}\in X_{\La}\text{ and }\beta(\la,x_{\G})\ast \la\ast x_{\G}\in X_{\G},$$
where uniqueness follows from the fact that $X_{\G}$ and $X_{\La}$ are fundamental domains for the $\G$- and the $\La$-actions respectively. The measurable maps $\alpha\colon \G\times X_{\La}\to \La$ and $\beta\colon \La\times X_{\G}\to \G$ are the \textit{cocycles} associated to this coupling.
\end{definition}

We now assume that $\G$ and $\La$ are finitely generated groups, with finite generating subsets $S_{\G}$ and $S_{\La}$, and we denote by $|.|_{S_{\G}}$ and $|.|_{S_{\La}}$ the associated word metric.\par
Given $p,q\in [0,+\infty]$, an $(\ld^p,\ld^q)$ \textit{measure equivalence coupling} from $\G$ to $\La$ is a measure equivalence coupling such that for every $\g\in\G$, the map $|\alpha(\g,.)|_{S_{\La}}\colon X_{\La}\to\N$ is $\ld^p$, and for every $\la\in\La$, the map $|\beta(\la,.)|_{S_{\G}}\colon X_{\G}\to\N$ is $\ld^q$. In particular, $\ld^0$ means that there is no requirement on the corresponding cocycle. We can more generally define the notion of $(\varphi,\psi)$\textit{-integrability} for any maps $\varphi,\psi\colon\R_+\to\R_+$. We will give the definition in Section~\ref{sec:Prel:ME}, where we define measure equivalence in the more general setting of unimodular locally compact second countable (lcsc) groups. For instance, $\ld^p$ is exactly $\varphi$-integrability when $\varphi(t)=t^p$.\par
In~\cite{shalomHarmonicAnalysisCohomology2004}, Shalom was the first to study quantitative measure equivalence, more precisely the case of $\ld^{\infty}$ measure equivalence, for amenable groups. His goal was to prove the quasi-isometry invariance of Betti numbers for nilpotent groups. We say that a measure equivalence coupling $(\Omega,X_{\G},X_{\La},\mu)$ is \textit{mutually cobounded} if there exist finite subsets $F_{\La}\subset\La$ and $F_{\G}\subset\G$ such that $X_{\G}\subset F_{\La}\cdot X_{\La}$ and $X_{\La}\subset F_{\G}\cdot X_{\G}$.

\begin{theorem}[\cite{shalomHarmonicAnalysisCohomology2004}]\label{th:shalom}
	Two countable amenable groups are quasi-isometric if and only if there exists an $\ld^{\infty}$ mutually cobounded measure equivalence coupling between them.
\end{theorem}

Historically, the notion of $\ld^p$ measure equivalence, for $p\ge 1$, was introduced in~\cite{baderIntegrableMeasureEquivalence2013}, and more generally $(\varphi,\psi)$-integrable measure equivalence was first defined in~\cite{delabieQuantitativeMeasureEquivalence2022} to study the weaker notion of $\ld^p$ orbit equivalence for $p<1$. For weaker assumptions than $\ld^{\infty}$, which geometric properties of amenable groups can be captured? In~\cite{austinIntegrableMeasureEquivalence2016}, Austin first proved that integrably measure equivalent groups of polynomial growth have bi-Lipschitz equivalent asymptotic cones.

\paragraph{Behaviour of volume growth.} Many finer rigidity results have been obtained, involving the volume growth. Given a finitely generated group $\G$, with a finite generating subset $S_{\G}$, the \textit{volume growth} with respect to $S_{\G}$ is the map $V_{S_{\G}}\colon\N\to\N$ defined by
$$V_{S_{\G}}(n)=|\{\g\in \G\mid |\g|_{S_{\G}}\leq n\}|.$$
It is straightforward to show that two different finite generating sets yield asymptotically equivalent growth functions, and the common asymptotic behaviour is denoted by $V_{\G}$.\par
For instance, $V_{\Z^d}(n)\simeq n^d$. More generally, a celebrated result of Gromov~\cite{gromovGroupsPolynomialGrowth1981} states that virtually nilpotent groups are exactly the groups with polynomial growth. Given non-trivial groups $F$ and $\G$, we define the wreath product $F\wr\G$ as
$$F\wr\G=\left (\bigoplus_{\G}{F}\right)\rtimes\G$$
where the action of $\G$ on the direct sum is induced by the action of $\G$ on itself by left translation. It is well known that if $F$ and $\G$ are finitely generated, then so is $F\wr\G$. If moreover $F$ is not trivial and $\G$ is infinite, then $V_{F\wr\G}(n)\simeq e^n$.\par
The notion of volume growth also exists in the case of compactly generated locally compact groups, see Section~\ref{sec:conventions}.\par
In the appendix of~\cite{austinIntegrableMeasureEquivalence2016}, Bowen proved an obstruction for integrable measure equivalence, using the notion of volume growth of finitely generated groups.

\begin{theorem}[{\cite[Theorem~B.2]{austinIntegrableMeasureEquivalence2016}}]
Let $\Gamma$ and $\Lambda$ be finitely generated groups. If $\Gamma$ and $\Lambda$ are $\ld^{1}$ measure equivalent, then $V_{\Gamma}(n)\simeq V_{\Lambda}(n)$.
\end{theorem}

For instance, $\Z^k$ and $\Z^d$ are integrably measure equivalent if and only if $k=d$. The results of Austin and Bowen show that integrable measure equivalent preserves significant coarse geometric invariants. It is therefore natural to wonder whether these rigidity results still hold for more general quantifications. In the wider setup of $(\varphi,\psi)$-integrability, Delabie, Koivisto, Le Maître and Tessera refined Bowen's result as follows.

\begin{theorem}[{\cite[Theorem~3.1]{delabieQuantitativeMeasureEquivalence2022}}]\label{th:dklmtgrowth}
Let $\Gamma$ and $\Lambda$ be finitely generated groups. Let $\varphi\colon\R_{+}\longrightarrow \R_{+}$ be an increasing and subadditive map. If there is a $(\varphi, \ld^{0})$-integrable measure equivalence coupling from $\Gamma$ to $\Lambda$, then 
\begin{equation*}
    V_{\Gamma}(n) \preccurlyeq V_{\Lambda}(\varphi^{-1}(n)).
\end{equation*}
\end{theorem}

For instance, given positive integers $k$ and $\ell$, there is no $(\ld^p,\ld^0)$ measure equivalence coupling from $\Z^{k+\ell}$ and $\Z^k$ if $p>\frac{k}{k+\ell}$.\par
Delabie, Llosa Isenrich and Tessera have recently brought a first extension to locally compact groups when $\varphi(t)=t^p$.

\begin{theorem}[{\cite[Theorem~A.5]{delabie$mathrmL^p$MeasureEquivalence2025}}]\label{th:growthdklmt}
    Let $G$ and $H$ be unimodular locally compact compactly generated groups and let $p\in ]0,1]$. If there exists an $(\ld^p,\ld^0)$-integrable measure equivalence coupling from $G$ to $H$, then
    $$V_G(n)\preccurlyeq V_H(n^{1/p}).$$
\end{theorem}

We extend this result to $(\varphi,\ld^0)$-measure equivalence, where $\varphi\colon\R_+\to\R_+$ is any increasing and subadditive map, for instance $\varphi(t)=\log{(1+t)}$.

\begin{theoremletter}[see Theorem~\ref{th:volumegrowth}]\label{th:volumegrowth intro}
    Let $G$ and $H$ be unimodular locally compact compactly generated groups and $\varphi\colon\R_+\to\R_+$ be an increasing and subadditive map. If there exists a $(\varphi,\ld^0)$-integrable measure equivalence coupling from $G$ to $H$, then
    $$V_G(n)\preccurlyeq V_H(\varphi^{-1}(n)).$$
\end{theoremletter}

\begin{remark}
We focus on unimodular groups because $\varphi$-integrability is undefined in the non-unimodular setting. Indeed, a definition of measure equivalence exists in the non-unimodular case (see e.g.~\cite[Definition~3.2]{koivistoMeasureEquivalenceNonunimodular2021}), but asks for the more general notion of non-singular actions. Another obstruction for an extension to non-unimodular groups will be the fact that isoperimetric profile, that we now deal with, does not contain any interesting information in the case of non-unimodular locally compact groups (we refer the reader to~\cite[comments before Section~5.2]{Tes13}).
\end{remark}

\paragraph{Behaviour of isoperimetric profiles.} A limitation of the previous rigidity results is that volume growth cannot distinguish between groups that get bigger and bigger. For instance, in the case of discrete groups, the following iterations of lamplighters
$$\Z/2\Z\wr(\Z/2\Z\wr\ldots (\Z/2\Z\wr\Z)\ldots)$$
all have exponential growth, but it is natural to expect that such groups are integrably measure equivalent only if they have the same number of iterations. The isoperimetric profiles enable us to distinguish them.\par
For a finitely generated group $\G$ with a finite generating set $S_{\G}$, and given $p\geq 1$, its $\ell^p$-\textit{isoperimetric profile} is the function $j_{p,\G}\colon \N\longrightarrow\R_{+}$ given by 
\begin{equation*}
    j_{p,\G}(n) \coloneq\sup_{\substack{f\colon \G\to\R_+\\|\supp{f}|\leq n}}{\frac{\|f\|_p}{\|\nabla_{\G} f\|_p}}
\end{equation*}
where the support of $f\colon \G\longrightarrow\R_+$ is $\supp{f}\coloneq\lbrace \g\in \G : f(\g)\neq 0\rbrace$ and the $\ell^p$-norm of its gradient is defined by 
\begin{equation*}
   \|\nabla_{\G} f\|_{p}^{p}\coloneq\sum_{\g\in \G,\; s\in S_{\G}}\left|f(\g)-f(s^{-1}\g)\right|^p. 
\end{equation*}
In the case $p=1$, this function is simply called \textit{isoperimetric profile}, and has a simpler definition (up to asymptotic equivalence), namely
\begin{equation*}
    j_{1,\G}(n) \simeq\sup_{|A|\le n}\frac{|A|}{|\partial_{\G} A|},
\end{equation*}
where $\partial_{\G}A \coloneq AS_{\G}\setminus A=\lbrace \g\in \G\setminus A : \exists s\in S_{\G}, \exists a\in A, \g=as\rbrace$ is the \textit{boundary} of $A$ in $\G$. It is well known that $j_{1,\G}$ is the generalized inverse of the \textit{F\o lner function}\par
\begin{equation*}
   \text{F\o  l}_\Gamma(k) =\inf\left\{|A|\colon  \frac{|\partial A|}{|A|}\leq 
1/k\right\}
\end{equation*}
first mentioned by Vershik~\cite{Vershik}. Finally a well-known inequality due to Coulhon and Saloff-Coste~\cite{coulhonIsoperimetriePourGroupes1993} relates $j_{1,\G}$ and $V_{\G}$:
$$j_{1,\G}(n)\preccurlyeq V_{\G}^{-1}(n),$$
where $V_\G^{-1}$ denotes the generalized inverse of the volume growth function. Here are examples of isoperimetric profiles for some groups:
\begin{itemize}
    \item $j_{p,\G}(n) \simeq n^{\frac{1}{d}}$ if $\G$ has polynomial growth of degree $d\ge 1$;
    \item $j_{p,\G}(n)\simeq \ln(n)$ for the Baumslag-Solitar group $\G=\text{BS}(1,k)$ for $k\ge 2$, or $\G=F\wr\Z$, where $F$ is a non-trivial finite group;
    \item $j_{p,\G}(n) \simeq \ln(n)$ for any polycyclic group $\G$ with exponential growth~\cite{Pit95, Pit00};
    \item $j_{1,F\wr \G}(n)\simeq (\ln(n))^{\frac{1}{d}}$ with $F$ finite, and $\G$ having polynomial growth of degree $d\ge 1$~\cite{erschlerIsoperimetricProfilesFinitely2003};
    \item Brieussel and Zheng~\cite[Theorem~1.1]{BZ21} give a large description of asymptotic behaviours for isoperimetric profiles. For any non-decreasing function $f\colon \R_{+}\to\R_{+}$ such that $t\mapsto \frac{t}{f(t)}$ is non-decreasing, there exists a finitely generated group $\G$ with exponential volume growth having isoperimetric profile $j_{p,\G}(n)\simeq \frac{\ln(n)}{f(\ln(n))}$.
\end{itemize}

Notice that the $\ell^p$-isoperimetric profile of a finitely generated group is bounded if and only if the group is non-amenable. Moreover its asymptotic behaviour is, somehow, a measurement of its amenability; the faster it goes to infinity, the "more amenable" the group is. The following result clearly demonstrates that quantitative measure equivalence accurately captures the geometry of amenable groups.

\begin{theorem}[{\cite[Theorem~1.1]{delabieQuantitativeMeasureEquivalence2022}}]\label{th:dklmtprofile}
Let $\varphi\colon\R_{+}\longrightarrow\R_{+}$ be a non-decreasing function. Let $\Gamma$ and $\Lambda$ be finitely generated groups. Assume that there exists a $(\varphi,\ld^0)$-integrable measure equivalence coupling from $\G$ to $\La$. Then
\begin{itemize}
    \item if $\varphi(t)=t^p$ with $p\geq 1$, then $j_{p,\La}(n) \preccurlyeq j_{p,\G}(n)$;
    \item if $t\mapsto \frac{t}{\varphi(t)}$ is non-decreasing, then $\varphi\circ j_{1,\La}(n) \preccurlyeq j_{1,\G}(n)$.
\end{itemize}
\end{theorem}

Delabie, Koivisto, Le Maître and Tessera~\cite{delabieQuantitativeMeasureEquivalence2022} also proved rigidity results for asymmetric notions of measure equivalence (measure subgroup, measure quotient, measure sub-quotient). In their works, they often use a crucial assumption on couplings, called being \textit{at most $m$-to-one}. To keep it simply, we define this notion in the context of measure equivalence: a coupling from $\G$ to $\La$ is at most $m$-to-one if the fibers of $\alpha(.,x)\colon\G\to\La$ have cardinality at most $m$ for almost every $x\in X_{\La}$.\par
In the same vein as Shalom's Theorem~\ref{th:shalom}, they proved the equivalence between the existence of an at most $m$-to-one $\ld^{\infty}$ measure subgroup from $\G$ to $\La$, and the existence of a regular embedding $\G\to\La$. By a \textit{regular embedding}, we mean a Lipschitz map $f\colon\G\to\La$ such that $\sup_{\la\in\La}{|f^{-1}(\{\la\})|}$ is finite. As a stunning consequence of the quantitative results of~\cite{delabieQuantitativeMeasureEquivalence2022}, the isoperimetric profiles are monotonous under regular embeddings.\par
It is worth noticing that the \enquote{at most $m$-to-one} assumption is essential in their adaptation of Shalom's theorem, since the bijections $\alpha(.,x)$, $x\in X_{\La}$, provided by the cocycle of an at most $m$-to-one $\ld^{\infty}$ measure subgroup are almost surely regular embedding, thus proving the existence of such maps from $\G$ to $\La$. However, for finer quantitative versions, the assumption seems avoidable. Although they could not get rid of it for rigidity results about quantitative measure sub-quotient, they manage to state Theorem~\ref{th:dklmtprofile} about quantitative measure equivalence without this assumption. We explain at the end of this introduction the basic ideas, and compare them with ours.\par
An extension of such quantitative results for asymmetric versions of measure equivalence was first brought by the second-named author~\cite{paucarIsoperimetricProfilesRegular2024}, where the assumption \enquote{at most $m$-to-one coupling} became \enquote{coarsely $m$-to-one coupling}. However the author could not avoid it to generalize Theorem~\ref{th:dklmtprofile} in the locally compact setting. We are now able to state such a generalization without this assumption.

\begin{theoremletter}[see Theorem~\ref{th:Lp}]\label{th:Lp intro}
    Let $G$ and $H$ be unimodular locally compact compactly generated groups and $p\geq 1$. If there exists an $(\ld^p,\ld^0)$ measure equivalence coupling from $G$ to $H$, then
    $$j_{p,H}(n)\preccurlyeq j_{p,G}(n).$$
\end{theoremletter}

\begin{theoremletter}[see Theorem~\ref{th:phi}]\label{th:phi intro}
    Let $G$ and $H$ be unimodular locally compact compactly generated groups and $\varphi\colon\R_+\to\R_+$ be a non-decreasing map such that $t\mapsto t/\varphi(t)$ is non-decreasing. If there exists a $(\varphi,\ld^0)$-integrable measure equivalence coupling from $G$ to $H$, then
    $$\varphi\circ j_{1,H}(n)\preccurlyeq j_{1,G}(n).$$
\end{theoremletter}

\paragraph{Integrability threshold.} \sloppy Given positive integers $k,\ell\geq 1$ and polynomial growth groups $\Gamma$ and $\Lambda$ of degrees $k+\ell$ and $k$ respectively, the obstructions provided in~\cite{delabieQuantitativeMeasureEquivalence2022} imply the following: there is no $(\mathrm{L}^p,\mathrm{L}^0)$ integrable measure equivalence coupling from $\Gamma$ to $\Lambda$ for any $p>\frac{k}{k+\ell}$. The existence of an $(\mathrm{L}^p,\mathrm{L}^0)$ integrable measure equivalence coupling for any $p<\frac{k}{k+\ell}$ have recently been proved in~\cite[Theorem~1.6]{delabie$mathrmL^p$MeasureEquivalence2025}. For $\Gamma=\Z^{k+\ell}$ and $\Lambda=\Z^{k}$, the result was already known~\cite[Theorem~1.9]{delabieQuantitativeMeasureEquivalence2022}.\par
Now the question is the existence of such a coupling for the threshold $p=\frac{k}{k+\ell}$. For more general groups, not necessarily of polynomial growth, we would like to know if the bounds of integrability given by Theorems~\ref{th:dklmtgrowth} and~\ref{th:dklmtprofile} can be reached. The first-named author has recently answered this question~\cite[Theorems~A,B~and~D]{correiaAbsenceQuantitativelyCritical2025} in the discrete setting and the following statements provide extensions to locally compact groups.

\begin{theoremletter}[see Theorem~\ref{th:ThresholdProfile}]\label{th:ThresholdProfile intro}
    Let $G$ and $H$ be unimodular locally compact compactly generated groups. Assume that there exist a non-decreasing function $f_G$ and an increasing function $f_H$ satisfying $f_G(n)\simeq j_{1,G}(n)$, $f_H(n)\simeq j_{1,H}(n)$ and the following assumptions as $t\to +\infty$:
		\begin{equation}\label{hyp1 prof intro}
			f_G(t)=o\left (f_H(t)\right ),
		\end{equation}
		\begin{equation}\label{hyp2 prof intro}
			\forall C>0,\ f_G(Ct)=O\left (f_G(t)\right ),
		\end{equation}
		\begin{equation}\label{hyp3 prof intro}
			\forall C>0,\ f_G\circ f_H^{-1}(Ct)=O\left (f_G\circ f_H^{-1}(t)\right ).
		\end{equation}
		Then there is no $(f_G\circ f_H^{-1},\ld^0)$-integrable measure equivalence coupling from $G$ to $H$.
\end{theoremletter}

\begin{theoremletter}[see Theorem~\ref{th:ThresholdGrowth}]\label{th:ThresholdGrowth intro}
    Let $G$ and $H$ be unimodular locally compact compactly generated groups. Assume that there exist two increasing functions $f_G$ and $f_H$ satisfying $f_G(n)\simeq V_G(n)$, $f_H(n)\simeq V_H(n)$ and the following assumptions as $t\to +\infty$:
		\begin{equation}\label{hyp1 growth intro}
			f_G^{-1}(t)=o\left (f_H^{-1}(t)\right ),
		\end{equation}
		\begin{equation}\label{hyp2 growth intro}
			\forall C>0,\ f_G^{-1}(Ct)=O\left (f_G^{-1}(t)\right ),
		\end{equation}
		\begin{equation}\label{hyp3 growth intro}
			\forall C>0,\ f_G^{-1}\circ f_H(Ct)=O\left (f_G^{-1}\circ f_H(t)\right ).
		\end{equation}
		Then there is no $(f_G^{-1}\circ f_H,\ld^0)$-integrable measure equivalence coupling from $G$ to $H$.
\end{theoremletter}

\paragraph{Main ingredients for the extensions to the locally compact setting.} We finally record the obstacles encountered in our work, and briefly describe the main tools.
\begin{itemize}
    \item For the discrete setting, the proofs in~\cite{delabieQuantitativeMeasureEquivalence2022} cannot be immediately adapted to a proof for the locally compact setting. To extend it, we had to find a slightly different approach that we now briefly describe.\par
    First of all, the proofs in the discrete setting are easier with the \enquote{at most $m$-to-one} assumption. Indeed, since Haar measures are simply counting measures in this context, it implies that the random maps $\alpha(.,x)\colon\G\to\La$ and $\beta(.,x)\colon\La\to\G$ are not far from being Haar-measure preserving up to a multiplicative constant (although they are not bijections). Then the strategy of the authors to get rid of this assumption is to average and thus prove in fact that this property of being at most $m$-to-one occurs with high probability, which is enough to get Theorem~\ref{th:dklmtprofile}.\par
    Our approach now consists in using Haar-measure preserving maps arising from cocycles. The example of orbit equivalence, namely measure equivalence with $X_{\G}=X_{\La}$, is particularly enlightening. Indeed, a much easier proof of Theorem~\ref{th:dklmtprofile} can be found in this setting, using the fact that the cocycles provide Haar-measure preserving $\alpha(.,x)\colon\G\to\La$ and $\beta(.,x)\colon\La\to\G$, since these are bijections and Haar measures are counting measures. Coming back to the general case of measure equivalence, the analogous property is the following: if $x$ lies in the intersection $X_{\G}\cap X_{\La}$ between the fundamental domains (up to translation, we can assume that the intersection has positive measure), then the maps $\alpha(.,x)$ and $\beta(.,x)$ induce bijections (inverses of each other) between sets $R_{X_{\G}\cap X_{\La}}^{\G}(x)$ and $R_{X_{\G}\cap X_{\La}}^{\La}(x)$ defined by
    $$R_{X_{\G}\cap X_{\La}}^{\G}(x)\coloneq \{\g\in\G\mid \g\cdot x\in X_{\G}\cap X_{\La}\}$$
    and similarly for $R_{X_{\G}\cap X_{\La}}^{\La}(x)$, where the symbol $\cdot$ refers to the induced actions $\G\curvearrowright X_{\La}$ and $\La\curvearrowright X_{\G}$. For the locally compact case, it is proved in~\cite[Proposition~A.1]{delabie$mathrmL^p$MeasureEquivalence2025}, using cross-sections, that we can assume without loss of generality that the fundamental domains intersect themselves (with positive measure) and that the random bijections $\alpha(.,x)\colon R_{X_{G}\cap X_{H}}^{G}(x)\to R_{X_{G}\cap X_{H}}^{H}(x)$ and $\beta(.,x)\colon R_{X_{G}\cap X_{H}}^{H}(x)\to R_{X_{G}\cap X_{H}}^{G}(x)$ preserve the Haar measures for almost every $x\in X_G\cap X_H$.\par
    We can in fact extend this property to almost every $x\in X_{\La}$: $\alpha(.,x)$  induces a Haar-measure preserving map from $R_{X_{\G}\cap X_{\La}}^{\G}(x)$ to its image (and similarly for $\beta(.,x)$ on $R_{X_{\G}\cap X_{\La}}^{\La}(x)$, for almost every $x\in X_{\G}$). To this end, we require an ergodic coupling, and the desired property is a direct consequence of the fact that $R_{X_{\G}\cap X_{\La}}^{\G}(x)$ (resp.~$R_{X_{\G}\cap X_{\La}}^{\La}(x)$) is nonempty for almost every $x\in X_{\La}$ (resp.~$x\in X_{\G}$). Note that ergodicity is not a restrictive assumption in the discrete setting, and even in the locally compact setting, as we will verify using the ideas from~\cite[Proposition~2.17~(ii)]{koivistoMeasureEquivalenceCoarse2021}. Further details are provided in Section~\ref{sec:preparatorylemma}.\par
    We thus want a new proof in the discrete setting, that only relies on this property of the cocycles, and that can be adapted in the locally compact setting. To use the properties of cocycles, the challenging part is to prove that we do not lose information when reducing $\G$ to $R_{X_{\G}\cap X_{\La}}^{\G}(x)$, given $x\in X_{\La}$, and $\La$ to $R_{X_{\G}\cap X_{\La}}^{\La}(x)$, given $x\in X_{\G}$. To this end, we look at quantities on average, similarly to the strategies of Delabie, Koivisto, Le Maître and Tessera. As an illustration, given a map $\theta\colon\G\to\R_+$, Fubini's theorem implies that
    $$\text{average}_{x\in X_{\La}}{\int_{R_{X_{\G}\cap X_{\La}}^{\G}(x)}{\theta(\g)\mathrm{d}\lambda_{\G}(\g)}}=\text{constant}\times\int_{\G}{\theta(\g)\mathrm{d}\lambda_{\G}(\g)}$$
    (we refer the reader to the last calculations in the proof of Proposition~\ref{prop:2} for instance), where $\lambda_{\G}$ refers to the Haar measure of $\G$, which is simply the counting measure. We can thus replace $\int_{\G}{\theta(\g)\mathrm{d}\lambda_{\G}(\g)}$ by $\int_{R_{X_{\G}\cap X_{\La}}^{\G}(x)}{\theta(\g)\mathrm{d}\lambda_{\G}(\g)}$ for a huge amount of elements $x\in x_{\G}$ (with a multiplicative error).
    
    \item We finally move on to Theorems~\ref{th:ThresholdProfile intro} and~\ref{th:ThresholdGrowth intro}. For finitely generated groups, we only need to check finitely many properties to prove that a cocycle is $\varphi$-integrable, namely the convergence of integrals indexed on a finite generating set. In the non-discrete setting, for compactly generated groups, not only do we have to check uncountably many properties: the convergence of integrals indexed on a compact generating set; but the supremum of these integrals over this set must also be finite (see Definition~\ref{def:IntCocycle}). Since the first-named author~\cite{correiaAbsenceQuantitativelyCritical2025} crucially used the finiteness of a generating set in the discrete setting, the generalization to the case of locally compact compactly generated groups is not immediate and requires more tools, such as the techniques in~\cite[Appendix A.2]{baderIntegrableMeasureEquivalence2013}.
\end{itemize}

\paragraph{Outline of the paper.} After recalling some preliminaries in Section~\ref{sec:Prel}, we present in Section~\ref{sec:preparatorylemma} a preparatory lemma for the main statements of the paper. The theorem on volume growth is proven in Section~\ref{sec:volgrowth}, the ones on isoperimetric profiles in Section~\ref{sec:isopprof}, and Section~\ref{sec:threshold} deals with the integrability thresholds.

\paragraph{Acknowledgements.}

We are very grateful to Romain Tessera for his valuable advice. We also thank François Le Maître for his insightful questions regarding the content of the paper, and Vincent Dumoncel for his careful reading.

\section{Preliminaries}\label{sec:Prel}

\subsection{Conventions and notations}\label{sec:conventions}

In this paper, groups $G$ and $H$ are always assumed to be compactly generated locally compact, namely there exists a compact subset $S_G$ of $G$ such that $G=\bigcup_{n\geq 0}{S_G^n}$, and similarly for $H$. Given a compactly generated locally compact group $G$, we can define the word length by
$$|g|_G=\min{\{n\geq 0\mid g\in S_G^n\}}$$
for every $g\in G$, which gives rise to the left-invariant word metric $(g,g')\mapsto |g^{-1}g'|_G$. Given $g\in G$ and $n\geq 0$, $B_G(g,n)$ will refer to the closed ball centered at $g$ and of radius $n$. Denoting by $\la_G$ the Haar measure of $G$, the \textit{volume growth} of $G$ is the map $V_G\colon\N\to\N$ defined by
$$V_G(n)\coloneq \lambda_G\left (B_G(1_G,n)\right )=\la_G\left (\{g\in G\mid |g|_G\leq n\}\right )$$
for every integer $n\geq 0$.\par
Given non-decreasing maps $\varphi,\psi\colon\R_+\to\R_+$, we say that $\psi$ \textit{dominates} $\varphi$, denoted $\varphi(t)\preccurlyeq\psi(t)$, if there exists a constant $C>0$ such that $\varphi(t)\leq C\psi(Ct)$ for sufficiently large real numbers $x$. The maps $\varphi$ and $\psi$ are said to be \textit{asymptotically equivalent}, denoted $\varphi(t)\simeq\psi(t)$, if $\varphi(t)\preccurlyeq\psi(t)$ and $\psi(t)\preccurlyeq\varphi(t)$, and the \textit{asymptotic behaviour} refers to the class modulo $\simeq$. We also use the following stronger notions of domination:
\begin{itemize}
    \item $\varphi(t)=O(\psi(t))$ means that there exists a constant $C>0$ such that $\varphi(t)\leq C\psi(t)$ for sufficiently large real numbers $t$;
    \item $\varphi(t)=o(\psi(t))$ means that for every $\varepsilon>0$, there exists $t_{\varepsilon}>0$ such that $\varphi(t)\leq \varepsilon\psi(t)$ for every $t\geq t_{\varepsilon}$.
\end{itemize}

It is well known that two compact generating sets $S_G$ and $S'_G$ of a compactly generated locally compact group $G$ provide word metrics which are bilipschitz equivalent~\cite[Proposition~4.B.4.(3)]{cornulierMetricGeometryLocally2016}. As a consequence, the volume growths associated to $S_G$ and $S'_G$ are asymptotically equivalent (the same holds for the notions of isoperimetric profile introduced below). Thus the reason why we forget the dependency on $S_G$ in the notations $|.|_G$ and $V_G$ is that we only care about the asymptotics of such quantities.

\subsection{Quantitative measure equivalence of locally compact groups}\label{sec:Prel:ME}

The definition of measure equivalence in the locally compact setting is the following. 

\begin{definition}\label{def:coupling}
    Let $G$ and $H$ be unimodular lcsc groups, with Haar measures $\la_G$ and $\la_H$ respectively. We say that $G$ and $H$ are \textit{measure equivalent} if there exist a measure space $(\Om,\mu)$ and finite measure spaces $(X_G,\nu_G)$ and $(X_H,\nu_H)$, together with measured isomorphisms
    $$i_G\colon (G\times X_G,\la_G\otimes\nu_G)\to (\Om,\mu),$$
    $$i_H\colon (H\times X_H,\la_H\otimes\nu_H)\to (\Om,\mu),$$
    such that the $G$- and $H$-actions defined by
    $$g\ast i_G(g',x_G)=i_G(gg',x_G),$$
    $$h\ast i_H(h',x_H)=i_H(hh',x_H)$$
    commute.
\end{definition}

Without loss of generality, we will consider $X_G$ and $X_H$ as subsets of $\Omega$, so that they are fundamental domains of the $G$- and $H$-actions on $(\Omega,\mu)$ (which are essentially free and preserve the measure), and $i_G$ and $i_H$ will be the maps defined by $i_G(g,x_G)=g\ast x_G$ and $i_H(h,x_H)=h\ast x_H$.\par
The quadruple $(\Om,X_G,X_H,\mu)$ is called a \textit{measure equivalence coupling} between $G$ and $H$. It induces actions on the fundamental domains, giving rise to cocycles.

\begin{definition}\label{def:cocycle}
		A measure equivalence coupling $(\Omega,X_G,X_H,\mu)$ between unimodular lcsc groups $G$ and $H$ induces a finite measure-preserving $G$-action on $(X_H,\nu_{H})$ in the following way: for every $g\in G$ and every $x_H\in X_H$, $g\cdot x_H\in X_H$ is defined by the identity
		$$(H\ast g\ast x_H)\cap X_H=\{g\cdot x_H\},$$
		it is unique since $X_H$ is a fundamental domain for the $H$-action on $\Om$.\par
		This also yields a \textit{cocycle} $\alpha\colon G\times X_H\to H$ uniquely defined by
		$$\alpha(g,x_H)\ast g\ast x_H=g\cdot x_H,$$
		or equivalently $\alpha(g,x_H)\ast g\ast x_H\in X_H$, for almost every $x_H\in X_H$ and every $g\in G$. We similarly define a finite measure-preserving $H$-action on $(X_G,\nu_{G})$ and the associated cocycle $\beta\colon H\times X_G\to G$.
	\end{definition}
	
	\begin{remark}
		The cocycle $\alpha\colon G\times X_H\to H$ satisfies the cocycle identity
		$$\forall g_1,g_2\in G,\ \forall x_H\in X_H,\ \alpha(g_1g_2,x_H)=\alpha(g_1,g_2\cdot x_H)\alpha(g_2,x_H),$$
        as well as $\beta\colon H\times X_G\to G$.
	\end{remark}

We now introduce the quantitative versions of measure equivalence. We first define the restrictions that we add on the cocycles.

\begin{definition}\label{def:IntCocycle}
Let $G$ and $H$ be two unimodular locally compact compactly generated groups. Let $G\curvearrowright (X,\nu)$ be a measure-preserving action on a finite measure space and let $c\colon G\times X\to H$ be an $H$-valued cocycle (i.e.~a $H$-valued map satisfying the cocycle identity). Given a non-decreasing map $\varphi\colon\R_+\to\R_+$ such that $\varphi(Ct)=O(\varphi(t))$ for every constant $C>0$, we say that the cocycle $c\colon G\times X\to H$ is $\varphi$\textit{-integrable} if
\begin{equation*}
    \sup_{g\in S_G}{\int_{X}{\varphi(|c(g,x)|_H)\mathrm{d}\nu(x)}} < +\infty
\end{equation*}
\end{definition}

Examples of non-decreasing maps $\varphi\colon\R_+\to\R_+$ satisfying $\varphi(Ct)=O(\varphi(t))$ for every constant $C>0$ are:
\begin{itemize}
    \item non-decreasing subadditive maps, for instance $\varphi(t)=t^p$ with $p\in ]0,1]$;
    \item non-decreasing maps $\varphi$ such that $t\mapsto t/\varphi(t)$ is non-decreasing, for instance $\varphi(t)=\log{(1+t)}$. Indeed, if $C\leq 1$, then we immediately get $\varphi(Ct)\leq\varphi(t)$, and if $C>1$, then we have $C/\varphi(Ct)\geq t/\varphi(t)$, namely $\varphi(Ct)\leq C\varphi(t)$.
\end{itemize}
Note that this definition does not depend on the choice of the compact generating set for $H$, by assumption on $\varphi$, whereas it seems to depend on the compact generating set $S_G$ for $G$. Two remarks are in order.
\begin{itemize}
    \item First, it is possible to replace $S_G$ by $S_G^{-1}$. This is a consequence of the formula $c(g,x)=c(g^{-1},g\cdot x)^{-1}$ (provided by the cocycle identity) and the invariance of the measure.
    \item Secondly, if $\varphi$ is subadditive, being $\varphi$-integrable is the same as saying that $\int_{X}{\varphi(|c(g,x)|_H)\mathrm{d}\nu(x)}$ is finite for all $g$ in $G$ without any uniform condition on the bound, by the same argument as in~\cite[Appendix~A.2]{baderIntegrableMeasureEquivalence2013} (see also the proof of Proposition~\ref{prop:ThresholdKeyProp} which yields the main ingredients). For instance, this is the case for $\varphi(t) = t^{p}$ with $0<p\leq1$. In particular, this implies that the notion of $\varphi$-integrability does not depend on $S_G$ for such maps $\varphi$.
\end{itemize}

Let us now define the quantitative versions of measure equivalence.

\begin{definition}
    Let $G$ and $H$ be two unimodular locally compact compactly generated groups. Let $\varphi,\psi\colon\R_+\to\R_+$ be non-decreasing maps satisfying $\varphi(Ct)=O(\varphi(t))$ and $\psi(Ct)=O(\psi(t))$ for every constant $C>0$. We say that $(\Om,X_G,X_H,\mu)$ is a $(\varphi,\psi)$\textit{-integrable measure equivalence coupling} from $G$ to $H$ if it is a measure equivalence coupling between $G$ and $H$ such that the associated cocycles $\alpha\colon G\times X_H\to H$ and $\beta\colon H\times X_G\to G$ satisfy the following: $\alpha$ is $\varphi$-integrable and $\beta$ is $\psi$-integrable.
\end{definition}

For $p>0$, we write $\mathrm{L}^p$ instead of $\varphi$ or $\psi$ if we consider the map $t\mapsto t^p$, and we write $\mathrm{L}^0$ when no requirement is made on the corresponding cocycle. Finally, a measure equivalence coupling is $\varphi$\textit{-integrable} if it is $(\varphi,\varphi)$-integrable.\par
Finally, since we will make use of~\cite[Proposition~A.1]{delabie$mathrmL^p$MeasureEquivalence2025} which requires non-discrete groups, we will have to reduce the problem to this setting. The following will be particularly useful.

\begin{lemma}\label{lem:nondiscrete}
    Let $G$ and $H$ be unimodular locally compact compactly generated groups, and $\varphi,\psi\colon\R_+\to\R_+$ be increasing maps satisfying $\varphi(Ct)=O(\varphi(t))$ and $\psi(Ct)=O(\psi(t))$ for every constant $C>0$. Given non-discrete compact groups $K_G$ and $K_H$, let $\tilde G\coloneq G\times K_G$ and $\tilde H\coloneq K\times K_H$. Then the following properties hold:
    \begin{itemize}
        \item any $(\varphi,\psi)$-integrable measure equivalence coupling from $G$ and $H$ gives rise to $(\varphi,\psi)$-integrable measure equivalence coupling from $\tilde G$ and $\tilde H$;
        \item for every $p\geq 1$, $j_{p,G}(n)\simeq j_{p,\tilde G}(n)$ and $j_{p,H}(n)\simeq j_{p,\tilde H}(n)$;
        \item $V_G(n)\simeq V_{\tilde G}(n)$ and $V_H(n)\simeq V_{\tilde H}(n)$.
    \end{itemize}
    In other words, in Theorems~\ref{th:volumegrowth intro},~\ref{th:Lp intro},~\ref{th:phi intro},~\ref{th:ThresholdProfile intro} and~\ref{th:ThresholdGrowth intro}, we can assume without loss of generality that the groups $G$ and $H$ are not discrete.
\end{lemma}

\begin{proof}
    Invariance of isoperimetric profiles and volume growth follows from the fact that $G$ and $\tilde G$ are quasi-isometric, as well as $H$ and $\tilde H$. Note that, by compactness, the Haar measures $\lambda_{K_G}$ and $\lambda_{K_H}$ of $K_G$ and $K_H$ are finite, this is a crucial point for the sequel. Let $(\Omega,X_G,X_H,\mu)$ be a $(\varphi,\psi)$-integrable measure equivalence coupling from $G$ to $H$, with the maps $i_G$ and $i_H$ as in Definition~\ref{def:coupling} and the cocycles $\alpha\colon G\times X_H\to H$ and $\beta\colon H\times X_G\to G$ as in Definition~\ref{def:cocycle}. A coupling between $\tilde G$ and $\tilde H$ is built as follows:
    \begin{itemize}
        \item $\tilde\Omega=\Omega\times K_G\times K_H$, $\tilde\mu=\mu\otimes\lambda_{K_G}\otimes\lambda_{K_H}$;
        \item $\tilde X_{\tilde G}=X_G\times K_H$, $\tilde\nu_{\tilde G}=\nu_G\otimes\lambda_{K_H}$;
        \item $\tilde X_{\tilde H}=X_H\times K_G$, $\tilde\nu_{\tilde H}=\nu_H\otimes\lambda_{K_G}$;
        \item $\tilde\imath_{\tilde G}\colon ((g,k_G),(x_G,k_H))\in \tilde G\times \tilde X_{\tilde G}\mapsto (i_G(g,x_G),k_G,k_H)\in\Omega$;
        \item $\tilde\imath_{\tilde H}\colon ((h,k_H),(x_H,k_G))\in \tilde H\times \tilde X_{\tilde H}\mapsto (i_H(h,x_H),k_G,k_H)\in\Omega$,
    \end{itemize}
    In particular, the commuting actions of $\tilde G$ and $\tilde H$ on $(\tilde\Omega,\tilde\mu)$ are defined by:
    \begin{align*}
        (g,k_G)\ast \tilde\imath_G((g',k'_G),(x_G,k_H))&=(i_G(gg',x_G),k_Gk'_G,k_H)\\
        &=(g\ast i_G(g',x_G),k_Gk'_G,k_H),
    \end{align*}
    \vspace{-1cm}
    \begin{align*}
        (h,k_H)\ast \tilde\imath_H((h',k'_H),(x_H,k_G))&=(i_H(hh',x_H),k_G,k_Hk'_H)\\
        &=(h\ast i_H(h',x_H),k_G,k_Hk'_H).
    \end{align*}
    It is straightforward to prove that it defines a measure equivalence coupling. Moreover, the cocycles $\tilde\alpha\colon \tilde G\times \tilde X_{\tilde H}\to \tilde H$ and $\beta\colon \tilde H\times \tilde X_{\tilde G}\to \tilde G$ are defined by:
    $$\tilde\alpha((g,k_G),(x_H,k'_G))=(\alpha(g,x_H),1_{K_G}),$$
    $$\tilde\beta((h,k_H),(x_G,k'_H))=(\beta(h,x_G),1_{K_H}),$$
    this immediately implies that $\alpha$ is $\varphi$-integrable and $\beta$ is $\psi$-integrable.
\end{proof}

\subsection{Isoperimetric profile of locally compact groups}\label{sec:Prel:profile}

The goal of this subsection is to collect the relevant properties of the isoperimetric profile for locally compact groups. To this end, the crucial point is the definition of the $\ld^p$\textit{-norm} of the gradient. Recall that for a countable group $\G$ admitting a finite generating set $S_{\G}$, it is defined as
$$\|\nabla_{\G} f\|_p^p=\sum_{s\in S_{\G}}{\sum_{\g\in \G}{|f(\g)-f(s^{-1}\g)|^p}}=\sum_{s\in S_{\G}}{\|f-\lambda(s)f\|_p^p}$$
for every $f\colon \G\to\R$ of finite support, where $\lambda\colon \G\curvearrowright \ld^p(\G)$ is the left-regular representation. Given $G$ a locally compact group with compact generating subset $S_{G}$, and with fixed Haar measure $\la_G$, the most natural generalization is
\begin{align*}
    \|\gint_G f\|_p^p &= \int_{s\in S_G} \|f-\la(s)f\|_p^p\ \mathrm{d}\lambda_G(s)
\end{align*}
for $f\in\ld^p(G)$ with support of finite measure. We have another possible definition of the $\ld^p$-norm of the gradient, given by
\begin{align*}
    \|\gsup_G f\|_p &= \sup_{s\in S_G} \|f-\la(s)f\|_p.
\end{align*}
This second definition appears in our calculations when bounding above some quantities (see the proof of Lemma~\ref{lem:3}), so we want it to be equivalent to the first one for the underlying isoperimetric profiles to be asymptotically equivalent. Note that we easily get $\|\gint_G f\|_p^p\leq\lambda(S_G)\|\gsup_G f\|_p^p$.\par
The comparison between these notions has already been done in~\cite[Propositions~7.1 and~7.2]{tesseraLargescaleSobolevInequalities2008}. Here we provide another proof.

\begin{proposition}
    There exists a constant $C>0$ such that
    $$\frac{1}{C}\|\gsup_G f\|_p\leq \|\gint_G f\|_p\leq C\|\gsup_G f\|_p$$
    for every $f\in\ld^p(G)$.
\end{proposition}

\begin{proof}
    Let us consider $Z^1(G,\lambda)$, the set of cocycles for the left-regular representation on $\ld^p(G)$, namely the set of maps $b\colon G\to\ld^p(G)$ satisfying the $1$-cocycle relation
    $$b(gh)=b(g)+\lambda(g)b(h)$$
    for every $g,h\in G$. By~\cite[Proposition~1.13]{lopezneumannGrowthCocyclesIsometric2025}, the following are equivalent norms on $Z^1(G,\lambda)$:
    $$\|b\|_{\mathrm{sup}}\coloneq\sup_{s\in S_G}{\|b(s)\|_p},$$
    $$\|b\|_{\mathrm{int}}\coloneq\left (\int_{s\in S_G}{\|b(s)\|_p^p\ \mathrm{d}\lambda(s)}\right )^{\frac{1}{p}}.$$
    The result then follows from the fact that, for every $f\in\ld^p(G)$, we have $\|\gint_G f\|_p=\|b\|_{\mathrm{int}}$ and $\|\gsup_G f\|_p=\|b\|_{\mathrm{sup}}$ with $b\colon G\to L^p(G)$ given by $b(g)= f-\lambda(g)f$. Note that $b$ lies in $Z^1(G,\lambda)$ since it is a coboundary.  
\end{proof}

From this, we deduce that the underlying isoperimetric profiles are asymptotically equivalent. Let us fix the notations in the following.

\begin{definition}
    Let $G$ be a compactly generated locally compact group, with a compact generating set $S_G$. Let $p\geq 1$. Then we define the following two notions of $\ld^p$\textit{-isoperimetric profiles}:
    \begin{align*}
        j^{\mathrm{sup}}_{p,G}(v) = \sup_{\substack{f\in\ld^p(G), \\ \la_G(\supp(f))\leq v} } \dfrac{\|f\|_p}{\|\gsup_G f\|_p},\\
        j^{\mathrm{int}}_{p,G}(v) = \sup_{\substack{f\in\ld^p(G),\\ \la_G(\supp(f))\leq v} } \dfrac{\|f\|_p}{\|\gint_G f\|_p}.
    \end{align*}
Since these two maps are asymptotically equivalent, we will denote by $j_{p,G}$ their asymptotic behaviour.
\end{definition}

Two compact generating sets give rise to asymptotically equivalent isoperimetric profiles. Finally, note that if $G$ is unimodular, this is also the same as defining the isoperimetric profile with respect to the right-regular representation.

\section{A preparatory lemma}\label{sec:preparatorylemma}

Given a measure equivalence coupling $(\Om,X_G,X_H,\mu)$ between locally compact groups $G$ and $H$, and given the induced actions $G\curvearrowright (X_H,\nu_H)$ and $H\curvearrowright (X_G,\nu_G)$ of the groups on the fundamental domains, we set
$$R_Y^{G}(x)\coloneq\{g\in G\mid g\cdot x\in Y\}$$
for every $Y\subset X_H$ and every $x\in X_H$, and similarly $R_Y^H(x)$ for every $Y\subset X_G$ and $x\in X_G$.\par
The goal of this section is to prove the following key lemma which will be useful for our main theorems.

\begin{lemma}\label{lem:HaarPreserving}
    Let $\varphi,\psi\colon\R_+\to\R_+$ be non-decreasing maps and let $G$ and $H$ be unimodular non-discrete locally compact groups. If there exists a $(\varphi,\psi)$-integrable measure equivalence coupling from $G$ to $H$, then there exists a measure equivalence coupling $(\Omega,X_G,X_H,\mu)$ satisfying the followings:
    \begin{enumerate}[label=(P\arabic*)]
        \item\label{item:assumption1} there exists a constant $C>0$ such that $\restr{\nu_G}{X_G\cap X_H}=C\cdot \restr{\nu_H}{X_G\cap X_H}$ and $\nu_G(X_G\cap X_H)=C\cdot \nu_H(X_G\cap X_H)>0$;
        \item the associated cocycles $\alpha\colon G\times X_H\to H$ and $\beta\colon H\times X_G\to G$ are respectively $\varphi$- and $\psi$-integrable (in particular the new measure equivalence coupling is also $(\varphi,\psi)$-integrable from $G$ to $H$);
        \item\label{item:assumption3} the map 
        $$\alpha(.,x)\colon (R_{X_G\cap X_H}^G(x),\la_G)\to (\alpha(R_{X_G\cap X_H}^G(x),x),\la_H)$$
        is measure preserving for almost every $x\in X_H$, as well as the map 
        $$\beta(.,x)\colon (R_{X_G\cap X_H}^H(x),\la_H)\to (\beta(R_{X_G\cap X_H}^H(x),x),\la_G)$$
        for almost every $x\in X_G$.
    \end{enumerate}
\end{lemma}
If Property~\textit{\ref{item:assumption3}} is required only for the points $x$ in the intersection $X_G\cap X_H$, this lemma is relatively straightforward in the case of countable groups. Indeed, up to some translation, the fundamental domains intersect non-trivially, and it is easy to verify that $\alpha(.,x)\colon R_{X_G\cap X_H}^G(x)\to \alpha(R_{X_G\cap X_H}^G(x),x)$ and $\beta(.,x)\colon R_{X_G\cap X_H}^H(x)\to \beta(R_{X_G\cap X_H}^H(x),x)$ are bijective if $x$ lies in the intersection $X_G\cap X_H$. Since the Haar measures are the counting measures in this case, we immediately get that these bijections are measure preserving. In the non-discrete case, it is much harder and the techniques in~\cite[Proposition~A.1]{delabie$mathrmL^p$MeasureEquivalence2025} will help us, it consists in using cross-sections to slightly modify the fundamental domains so that Property~\textit{\ref{item:assumption3}} in the lemma hold for every $x$ in the intersection $X_G\cap X_H$.\par
To get Property~\textit{\ref{item:assumption3}} for almost every $x\in X_H$ for $\alpha$, and $x\in X_G$ for $\beta$, and not only for $x$ in the intersection, we require the induced actions to be ergodic, ensuring that almost every orbit in $X_G$ (resp.~in $X_H$) visits the intersection $X_G\cap X_H$. To this end, we follow the proof of~\cite[Proposition~2.17~(ii)]{koivistoMeasureEquivalenceCoarse2021}.\par
Finally, we must verify that $(\varphi,\psi)$-integrability is preserved under the transformations we apply. First, for the transformations used in~\cite[Proposition~2.17~(ii)]{koivistoMeasureEquivalenceCoarse2021}, preservation follows from the ergodic decomposition theorem. Secondly, in~\cite[Proposition~A.1]{delabie$mathrmL^p$MeasureEquivalence2025}, the new fundamental domains are obtained by translating the previous ones by elements of a compact subset of $G\times H$, ensuring a uniform bound
on the difference between the length norms of the new and original cocycles.

\begin{proof}[Proof of lemma~\ref{lem:HaarPreserving}]
    Let $(\Omega,X_G,X_H,\mu)$ be a $(\varphi,\psi)$-integrable measure equivalence coupling from $G$ to $H$. As in the definition, we denote by $i_G$, $i_H$ the measured isomorphisms
    $$i_G\colon (G\times X_G,\la_G\otimes\nu_G)\to(\Omega,\mu),$$
    $$i_H\colon (H\times X_H,\la_H\otimes\nu_H)\to(\Omega,\mu),$$
    where $\nu_G$ and $\nu_H$ are the finite measures on $X_G$ and $X_H$.\par
    \textbf{Step 1: assuming that $\mu$ is ergodic.} Let us first follow the proof of \cite[Proposition~2.17~(ii)]{koivistoMeasureEquivalenceCoarse2021}. We have to be careful with the notations, since our notations $X_H,X_G,\mu,\nu_G,\nu_H, i_G, i_H$ refer to the notations $X,Y,\eta,\mu,\nu,i,j$ in~\cite{koivistoMeasureEquivalenceCoarse2021}. By the Ergodic Decomposition Theorem, there is a standard probability space $(Z,\zeta)$ and a familly $((\nu_H)_z)_{z\in Z}$ of ergodic measures for the induced action $G\curvearrowright X_H$, such that
    \begin{equation}\label{eq:desintegration}
    \nu_H(A)=\int_{Z}{(\nu_H)_z(A)\ \mathrm{d}\zeta(z)}
    \end{equation}
    for every measurable subset $A\subset X_H$. We then set $\mu_z\coloneq (i_H)_{\star}[\la_H\otimes(\nu_H)_z]$, this is a measure on $\Omega$. It is shown in the proof of~\cite[Proposition~2.17~(ii)]{koivistoMeasureEquivalenceCoarse2021} that the measures $\mu_z$ are $\sigma$-finite and ergodic (with respect to the $(G\times H)$-action on $\Omega$) for almost all $z\in Z$, and give rise to measures $(\nu_G)_z$ on $X_G$ such that
    \begin{itemize}
        \item for every $z\in Z$, $\la_G\otimes (\nu_G)_z=(i_G^{-1})_{\star}\mu_z$;
        \item for almost every $z\in Z$, $(\nu_G)_z$ is a finite measure and is ergodic for the induced action $H\curvearrowright X_G$.
    \end{itemize}
    Given $z$ in a conull subset of $Z$ (a subset on which the above properties hold), we now consider the measure equivalence coupling $(\Omega,\mu_z,X_G,X_H)$, with the finite measures $(\nu_G)_z$ and $(\nu_H)_z$ on $X_G$ and $X_H$. By~\eqref{eq:desintegration}, we have
    $$\int_{X_H}{\varphi(|\alpha(g,x)|_H)\ \mathrm{d}\nu_H}=\int_{Z}{\left(\int_{X_H}{\varphi(|\alpha(g,x)|_H)\ \mathrm{d}(\nu_H)_z(x)}\right)\ \mathrm{d}\zeta(z)},$$
    so the cocycle $\alpha$ is $\varphi$-integrable with respect to $(\nu_H)_z$ for almost every $z\in Z$. We can also find a desintegration for $\nu_G$, similarly to~\eqref{eq:desintegration}, and prove that the other cocycle $\beta$ is $\psi$-integrable  with respect to $(\nu_G)_z$ for almost every $z\in Z$. So the couplings $(\Omega,\mu_z,X_G,X_H)$ are $(\varphi,\psi)$-integrable for almost all $z\in Z$.\par
    \textbf{Step 2: Property~\textit{\ref{item:assumption3}} for $x$ in the intersection $X_G\cap X_H$.} Let us fix $z$ in a conull subset of $Z$ for which the desired properties hold. By~\cite[Proposition~A.1]{delabie$mathrmL^p$MeasureEquivalence2025}, 
    we can choose new fundamental domains $X'_{G,z}$ and $X'_{H,z}$ satisfying Property~\textit{\ref{item:assumption1}}, and so that the new cocycles $\alpha'_z$ and $\beta'_z$ satisfy Property~\textit{\ref{item:assumption3}} for every $x\in X'_{G,z}\cap X'_{H,z}$. Furthermore, these cocycles are $\varphi$- and $\psi$-integrable, since their norms are bounded above by the norms of the previous cocycles with an additive term (see the proof of~\cite[Proposition~A.1]{delabie$mathrmL^p$MeasureEquivalence2025}).\par
    \textbf{Step 3: ergodicity for the induced actions.} For every $z$ in a conull subset of $Z$, we have built a measure equivalence coupling $(\Omega,\mu_z,X'_{G,z},X'_{H,z})$ where $\mu_z$ is ergodic and the associated cocycles satisfy the following: for almost every $y\in X'_{G,z}\cap X'_{H,z}$, the maps
    \begin{equation}\label{eq:map1}
    \alpha_z'(.,y)\colon (R_{X'_{G,z}\cap X'_{H,z}}^G(y),\la_G)\to (\alpha_z'(R_{X'_{G,z}\cap X'_{H,z}}^G(y),y),\la_H)
    \end{equation}
    and
        \begin{equation}\label{eq:map2}
        \beta_z'(.,y)\colon (R_{X'_{G,z}\cap X'_{H,z}}^H(y),\la_H)\to (\beta_z'(R_{X'_{G,z}\cap X'_{H,z}}^H(y),y),\la_G)
        \end{equation}
        are measure preserving, and the goal is now to extend it to $y\in X_H$ (for $\alpha_z'$) and to $y\in X_G$ (for $\beta_z'$), using ergodicity. Coming back to Step~1 of this proof, or equivalently to the proof of~\cite[Proposition~2.17~(ii)]{koivistoMeasureEquivalenceCoarse2021}, we know that the ergodic measure $\mu_z$ gives rise to ergodic finite measures $(\nu'_H)_z$ and $(\nu'_G)_z$ for the induced actions $G\curvearrowright X'_{H,z}$ and $H\curvearrowright X'_{G,z}$. To conclude, for almost every $x\in X'_{H,z}$, the set $R_{X'_{G,z}\cap X'_{H,z}}^G(x)$ is not empty and we pick a point $g_x$. Then for every $g\in R_{X'_{G,z}\cap X'_{H,z}}^G(x)$, the cocycle identity gives
        $$\alpha_z'(g,x)=\alpha(gg_x^{-1},g_x\cdot x)\alpha(g_x,x)$$
        where $g_x\cdot x$ lies in the intersection $X'_{G,z}\cap X'_{H,z}$. It finally follows from unimodularity, the equality $\left (R_{X'_{G,z}\cap X'_{H,z}}^G(x)\right )g_x^{-1}=R_{X'_{G,z}\cap X'_{H,z}}^G(g_x\cdot x)$ and the fact that the map~\eqref{eq:map1} (with $y=g_x\cdot x$) is measure preserving, that 
        $$\alpha_z'(.,x)\colon (R_{X'_{G,z}\cap X'_{H,z}}^G(x),\la_G)\to (\alpha_z'(R_{X'_{G,z}\cap X'_{H,z}}^G(x),x),\la_H)$$
        is measure preserving. The same holds for $\beta_z'$. This completes the proof.
\end{proof}

\section{Behaviour of volume growth}\label{sec:volgrowth}

In this section, we prove Theorem~\ref{th:volumegrowth intro} that we now recall.

\begin{theorem}\label{th:volumegrowth}
    Let $G$ and $H$ be unimodular locally compact compactly generated groups and $\varphi\colon\R_+\to\R_+$ be a non-decreasing and subadditive map. If there exists a $(\varphi,\ld^0)$-integrable measure equivalence coupling from $G$ to $H$, then
    $$V_G(n)\preccurlyeq V_H(\varphi^{-1}(n)).$$
\end{theorem}

We begin with a simple result, found for instance in~\cite[Proof of Lemma~B.11]{austinIntegrableMeasureEquivalence2016}.

\begin{lemma}\label{lem:returntime}
    Let $G\act (X,\nu)$ be a measure preserving action of a locally compact group on a probability space. Let $Y$ be a measurable subset of $X$, of positive measure. Then we have
    $$\int_{X}{\frac{\lambda_G(R_Y^G(x)\cap B_G(1_G,n))}{V_G(n)}\ d\nu(x)}=\nu(Y).$$
\end{lemma}

\begin{proof}
The proof follows directly from the invariance of the Haar measure:
    \begin{align*}
        &\int_{X}{\lambda_G(R_Y^G(x)\cap B_G(1_G,n))\ d\nu(x)}\\
        =&\int_X{\int_{B_G(1_G,n)}{\mathds{1}_{x\in g^{-1}Y}\ d\lambda_G(g)}\ d\nu(x)}\\
        =&\int_{B_G(1_G,n)}{\int_X{\mathds{1}_{x\in g^{-1}Y}\ d\nu(x)}\ d\lambda_G(g)}\\
        =&\int_{B_G(1_G,n)}{\nu(g^{-1}Y)\ d\lambda_G(g)}\\
        =&\int_{B_G(1_G,n)}{\nu(Y)\ d\lambda_G(g)}\\
        =&\lambda_G({B_G(1_G,n)})\nu(Y)\\
        =&V_G(n)\nu(Y),
    \end{align*}
    which completes the result.
\end{proof}

\begin{proof}[Proof of Theorem~\ref{th:volumegrowth}]
Without loss of generality, we assume that the groups are non-discrete (using Lemma~\ref{lem:nondiscrete}) and that the properties given by Lemma~\ref{lem:HaarPreserving} hold. Let us also assume that $\nu_H(X_H)=1$. Let $Y\coloneq X_G\cap X_H$. Let $K,K'>0$ be positive constants such that
$$\mu(Y)+\left (1-\frac{1}{K}\right )\left (1-\frac{1}{K'}\right )-1>0.$$
By~\cite[Lemma~2.24]{delabieQuantitativeMeasureEquivalence2022} which also holds in our framework, there exists a constant $C>0$ such that for every $g\in G$,
    $$\int_{X_H}{\varphi(|\alpha(g,x)|_H)d\nu_H(x)}\leq C|g|_G.$$
    We set
    $$X_1\coloneq\left \{x\in X_H\ \colon \int_{B_G(1_G,n)}{\frac{\varphi(|\alpha(g,x)|_H)}{|g|_G}d\la_G(g)}\leq CKV_G(n)\right \}$$
    and for every $x\in X_H$,
    $$G_x\coloneq\{g\in B_G(1_G,n)\mid\varphi(|\alpha(g,x)|_H)\leq CKK'|g|_G\}.$$
    By Markov's inequality on the probability space $(X_H,\nu_H)$, we have
    \begin{equation}\label{eq:boundX1}
        \nu_H(X_1)\geq 1-\frac{1}{K}.
    \end{equation}
    Then Markov's inequality on the probability space $\left (B_G(1_G,n),\frac{\la_G}{V_G(n)}\right )$ gives for every $x\in X_1$,
    \begin{equation}\label{eq:boundGx}
    \la_G(G_x)\geq \left (1-\frac{1}{K'}\right )V_G(n).
    \end{equation}
    To prove the theorem, we now derive bounds of $I\coloneq\displaystyle\int_{X_H}{\la_G(R_Y(x)\cap G_x)d\nu_H(x)}$.\par
    On the one hand, using the inclusion
        $$R_Y(x)\cap B_G(1_G,n)\subset \left (R_Y(x)\cap G_x\right )\cup (B_G(1_G,n)\setminus G_x)$$
        for every $x\in X_H$, we have
        \begin{align*}
            I&\geq \int_{X_H}{\left (\la_G(R_Y(x)\cap B_G(1_G,n))+\la_G(G_x)-\la_G(B_G(1_G,n))\right )d\nu_H(x)}\\
            &\geq \int_{X_H}{\la_G(R_Y(x)\cap B_G(1_G,n))d\nu_H(x)} +\int_{X_1}{\la_G(G_x)d\nu_H(x)} -\int_{X_H}{V_G(n)d\nu_H(x)}\\
            &\geq V_G(n)\underbrace{\left (\nu_H(Y)+\left (1-\frac{1}{K}\right )\left (1-\frac{1}{K'}\right )-1\right )}_{\eqcolon K''>0},
        \end{align*}
        where we also apply Lemma~\ref{lem:returntime} and Equations~\eqref{eq:boundX1} and~\eqref{eq:boundGx}.\par
        On the other hand, using the definition of $G_x$ and Property~\textit{\ref{item:assumption3}} of Lemma~\ref{lem:HaarPreserving}, we get
        \begin{align*}
            I&\leq\int_{X_H}{\int_{R_Y^G(x)}{\mathds{1}_{\alpha(g,x)\in B_H(1_H,\varphi^{-1}(CKK'n))}\ d\la_G(g)}\ d\nu_H(x)}\\
            &=\int_{X_H}{\int_{\alpha(R_Y^G(x),x)}{\mathds{1}_{h\in B_H(1_H,\varphi^{-1}(CKK'n))}\ d\la_H(h)}\ d\nu_H(x)}\\
            &\leq V_H(\varphi^{-1}(CKK'n)).
        \end{align*}
        Hence $K''V_G(n)\leq V_H(\varphi^{-1}(CKK'n))$, completing the proof.
\end{proof}

\section{Behaviour of isoperimetric profiles}\label{sec:isopprof}

The goal of this section is to prove Theorem~\ref{th:Lp intro} and~\ref{th:phi intro}, that we now recall.

\begin{theorem}\label{th:Lp}
    Let $G$ and $H$ be unimodular locally compact compactly generated groups and $p\geq 1$. If there exists an $(\ld^p,\ld^0)$-integrable measure equivalence coupling from $G$ to $H$, then
    $$j_{p,H}(n)\preccurlyeq j_{p,G}(n).$$
\end{theorem}

\begin{theorem}\label{th:phi}
    Let $G$ and $H$ be unimodular locally compact compactly generated groups and $\varphi\colon\R_+\to\R_+$ be a non-decreasing map such that $t\mapsto t/\varphi(t)$ is non-decreasing. If there exists a $(\varphi,\ld^0)$-integrable measure equivalence coupling from $G$ to $H$, then
    $$\varphi\circ j_{1,H}(n)\preccurlyeq j_{1,G}(n).$$
\end{theorem}

\subsection{Outlines of the proofs}\label{sec:outlines}

As for Theorem~\ref{th:volumegrowth intro}, we assume without loss of generality that the groups are non-discrete (using Lemma~\ref{lem:nondiscrete}) and the properties given by Lemma~\ref{lem:HaarPreserving} hold. \par
To prove Theorems~\ref{th:Lp} and~\ref{th:phi}, we fix a map $f\colon H\to\R$ whose support has finite measure. Given an integer $v\geq 0$, $f$ will be viewed as a map almost realising $j_{1,H}^{\mathrm{sup}}(v)$, namely $\la_H(\supp f)\leq v$ and $\frac{\|f\|_p}{\|\gsup_H f\|_p}$ is almost equal to $j_{1,H}(v)$.\par
From this map $f$ and the measure equivalence coupling, we deduce, via the identification $\Omega\simeq X_H\times H$, a new map $\tilde f\colon\Om\to\R$ defined by
$$\tilde f(h\ast x)=f(h^{-1})$$
for every $h\in H$ and $x\in X_H$. Moreover, from the identification $\Omega\simeq X_G\times G$, we get for every $x\in X_G$ a map $f_x\colon G\to\R$ defined by
$$f_x(g)=\tilde f (g\ast x)$$
for every $g\in G$.\par
Given $x\in X_G$, there exists another formula for the definition of $f_x$. Indeed, by definition, given $g\in G$, $f_x(g)$ is equal to $f(h^{-1})$ for some $h\in H$ satisfying $g\ast x\in h\ast X_H$. There exists a unique $h_x\in H$ such that $h_x\ast x\in X_H$, so $h^{-1}\ast g\ast x$ is a point in $X_H$ which is also equal to $h^{-1}h_x^{-1}\ast g\ast h_x\ast x$, meaning that $h^{-1}h_x^{-1}=\alpha(g,h_x\ast x)$. To conclude, $f_x$ is defined by
\begin{equation}\label{eq:closedformula}
    f_x(g)=f(\alpha(g,h_x\ast x)h_x)
\end{equation}
for every $g\in G$.\par
To sum up, a map $f\colon H\to\R$ with support of finite measure gives rise to a random map $f_x\colon G\to\R$ (random in $x\in X_G$ where $X_G$ is endowed with the finite measure $\nu_G$), given by formula~\eqref{eq:closedformula}. Furthermore, $f$ and the maps $f_x$ are also related by a map $\tilde{f}\colon\Omega\to\R$ coming from the measure equivalence coupling.

\begin{itemize}
    \item[\textbf{Step 1.}] By Proposition~\ref{prop:1}, the mean value of $\la_G(\supp f_x)$ is bounded above by $\la_H(\supp f)$ with some multiplicative constant.
    \item[\textbf{Step 2.}] Proposition~\ref{prop:2} is about the $\ld^p$-norm of $f_x$: its mean value is bounded below by $\|f\|_p$.
    \item[\textbf{Step 3.}] Propositions~\ref{prop:3} and~\ref{prop:4} finally deal with the mean of $\|\gint_G f_x\|_p^p$ and give an upper bound with $\|\gsup_H f\|_p^p$. These propositions are decomposed in three lemmas which use the function $\tilde{f}$, while the statements in the previous steps only use the formula~\eqref{eq:closedformula}. Lemma~\ref{lem:1} states that the mean value of $\|\gint_G f_x\|_p^p$ is exactly
    $$\|\gint_G \tilde{f}\|_p^p\coloneq\int_{S_G}{\int_{\Omega}{\left |\tilde f(s^{-1}\ast\omega)-\tilde f(\omega)\right |^p\mathrm{d}\mu(\omega)\ \mathrm{d}\lambda_G(s)}},$$
    and Lemmas~\ref{lem:2} and~\ref{lem:3} provides an upper bound of $\|\gint_G \tilde{f}\|_p^p$ given by $\|\gsup_H f\|_p^p$.
\end{itemize}
Theorems~\ref{th:Lp} and~\ref{th:phi} then follow from a pigeonhole principle.

\subsection{Proof of the theorems}

We follow the steps described in Section~\ref{sec:outlines} to prove Theorems~\ref{th:Lp} and~\ref{th:phi}. In the following propositions and lemmas, we assume that the groups are non-discrete (using Lemma~\ref{lem:nondiscrete}) and the properties given by Lemma~\ref{lem:HaarPreserving} hold, and we fix a map $f\colon H\to\R$ giving rise to a random map $f_x\colon G\to\R$ as described above.

\begin{proposition}\label{prop:1}
    With the same conventions and notations as in Section~\ref{sec:outlines}, there exists a positive constant $K$ (independent of $f$) such that
    $$\int_{X_G}{\la_G(\supp f_x)\ d\nu_G(x)}\leq K\la_H(\supp f).$$
\end{proposition}

\begin{proof}
    Let $x\in X_G$. Given $g\in G$ such that $f_x(g)\not=0$, the relation between $f$ and $f_x$ (see the paragraph before the formula~\eqref{eq:closedformula}) implies that there exists $h\in\supp f$ such that $h\ast g\ast x\in X_H$, so $x\in g^{-1}\ast (\supp f)^{-1}\ast X_H$. We thus get
    \begin{align*}
        \int_{X_G}{\la_G(\supp f_x)\ d\nu_G(x)}&=\int_{X_G}{\int_{G}{\mathds{1}_{f_x(g)\not=0}\ d\la_G(g)}\ d\nu_G(x)}\\
        &\leq\int_{X_G}{\int_{G}{\mathds{1}_{x\in g^{-1}\ast (\supp f)^{-1}\ast X_H}\ d\la_G(g)}\ d\nu_G(x)}\\
        &=\int_G{\nu_G(X_G\cap (g^{-1}\ast (\supp f)^{-1}\ast X_H))\ d\la_G(g)}\\
        &=\mu((\supp f)^{-1}\ast X_H)\\
        &=\la_H((\supp f)^{-1})\ \nu_H(X_H)\\
        &=\la_H(\supp f)\ \nu_H(X_H),
    \end{align*}
    where the last equality uses unimodularity.
\end{proof}

\begin{proposition}\label{prop:2}
    With the same conventions and notations as in Section~\ref{sec:outlines}, there exists a positive constant $K'$ (independent of $f$) such that for every $p\geq 1$,
    $$\int_{X_G}{\|f_x\|_p^p\ d\nu_G(x)}\geq K'\|f\|_p^p$$
\end{proposition}

\begin{proof}
    Let $x\in X_G$ and $g\in G$. Then we have $h\ast g\ast x\in X_H$ with $h=\alpha(g,h_x\ast x)h_x$. Let us denote $Y\coloneq X_G\cap X_H$ and let us more particularly assume that $h\ast g\ast x$ lies in $Y$. On the one hand, the definition of the induced action $H\act (X_G,\nu_G)$ implies $h\ast g\ast x=h\cdot x$. On the other hand, we get from the definition of the induced action $G\act (X_H,\nu_H)$ that $h\ast g\ast x=g\cdot (h_x\ast x)$. We thus get $g\in R_Y^G(h_x\ast x)\iff h\in R_Y^H(x)$, in other words 
    $$\left\{\alpha\left (g,h_x\ast x\right )h_x\mid g\in R_Y^G(h_x\ast x)\right\}=R_Y^H(x).$$
    Using the fact that the surjective map
    $$\alpha(.,h_x\ast x)h_x\colon (R_{Y}^G(h_x\ast x),\la_G)\to (R_Y^H(x),\la_H)$$
    is measure preserving (by the properties of Lemma~\ref{lem:HaarPreserving}), we thus get
    \begin{align*}
        \int_{X_G}{\|f_x\|_p^p\ d\nu_G(x)}&=\int_{X_G}{\int_{G}{|f(\alpha(g,h_x\ast x)h_x)|^p\ d\la_G(g)}\ d\nu_G(x)}\\
        &\geq \int_{X_G}{\int_{R_Y^G(h_x\ast x)}{|f(\alpha(g,h_x\ast x)h_x)|^p\ d\la_G(g)}\ d\nu_G(x)}\\
        &=\int_{X_G}{\int_{R_Y^H(x)}{|f(h)|^p\ d\la_H(h)}\ d\nu_G(x)}\\
        &=\int_{X_G}{\int_{H}{|f(h)|^p\mathds{1}_{x\in h^{-1}\cdot Y}\ d\la_H(h)}\ d\nu_G(x)}\\
        &=\int_{H}{|f(h)|^p\nu_G(h^{-1}\cdot Y)\ d\la_H(h)}\\
        &=\nu_G(Y)\|f\|_p^p
    \end{align*}
    and we are done.
\end{proof}

We now use $\tilde f\colon\Om\to\R$ which relates $f$ and the maps $f_x$ (for $x\in X_G$) via the measure equivalence coupling. Recall that the quantity $\|\gint_G \tilde{f}\|_p$ is defined by
$$\|\gint_G \tilde{f}\|_p^p\coloneq\int_{S_G}{\int_{\Omega}{\left |\tilde f(s^{-1}\ast\omega)-\tilde f(\omega)\right |^p\mathrm{d}\mu(\omega)\ \mathrm{d}\lambda_G(s)}}.$$

\begin{lemma}\label{lem:1}
With the same conventions and notations as in Section~\ref{sec:outlines}, we have for every $p\geq 1$,
    $$\int_{X_G}{\|\gint_Gf_x||_p^p\ d\nu_G(x)}=\|\gint_G \tilde f\|_p^p$$
\end{lemma}

\begin{proof}
    It is simply a change of variable $(g,x)\mapsto g\ast x$ from $(G\times X_G,\la_G\otimes\nu_G)$ to $(\Om,\mu)$.
\end{proof}

\begin{lemma}\label{lem:2}
With the same conventions and notations as in Section~\ref{sec:outlines}, we have for every $p\geq 1$,
    $$\|\gint_G\tilde f\|_p^p=\int_{S_G}{\int_{X_H}{\|\lambda(\alpha(s^{-1},x)^{-1}) f-f\|_p^p\ d\nu_H(x)}\ d\la_G(s)}$$
\end{lemma}

\begin{proof}
    The change of variable $h\ast x\mapsto (h,x)$ from $(\Om,\mu)$ to $(H\times X_H,\la_H\otimes\nu_H)$ gives
    $$\|\gint_G \tilde f\|_p^p=\int_{S_G}{\int_{H\times X_H}{|\tilde f(s^{-1}\ast h\ast x)-\tilde f(h\ast x)|^p\ d(\la_H\otimes\nu_H)(h,x)}\ d\la_G(s)}.$$
    Given $s\in S_G$ and $x\in X_H$, we have $s^{-1}\ast x=\alpha(s^{-1},x)^{-1}\ast (s^{-1}\cdot x)$. Now using the definition of $\tilde f$ from $f$, we obtain
    $$\tilde f(s^{-1}\ast h\ast x)=\tilde f(h\alpha(s^{-1},x)^{-1}\ast (s\cdot x))=f(\alpha(s^{-1},x)h^{-1})$$
    and
    $$\tilde f(h\ast x)=f(h^{-1})$$
    since $x$ and $s^{-1}\cdot x$ lie in $X_H$. Unimodularity of $H$ now implies that $h\mapsto h^{-1}$ preserves the Haar measure $\la_H$ and we are done.
\end{proof}

\begin{lemma}\label{lem:3}
    With the same conventions and notations as in Section~\ref{sec:outlines}, we have for any $p\geq 1$ and any $h$ in $H$,
    \begin{equation*}
        ||\lambda(h^{-1})f-f||_p \leq |h|_H ||\gsup_H f||_p
    \end{equation*}
\end{lemma}

\begin{proof}
 By the definition of $|h|_H\eqcolon n$, we have that $h^{-1}=s_1 s_2\ldots s_n$ where each $s_i$ belongs to $S_H\cup S_H^{-1}$ for $i\in\{1,\cdots, n\}$. Then it is clear that:
 \begin{align*}
     ||\lambda(h^{-1}) f-f||_p &\leq \sum_{i=0}^{n-1} ||\lambda(s_1\cdots s_{i+1}) f- \lambda(s_1\cdots s_{i}) f||_p \\
     &= \sum_{i=0}^{n-1} || \lambda(s_{i+1}) f- f||_p\\
     &\leq n ||\gint_H f||_p= |h|_H ||\gint_H f||_p,
 \end{align*} 
 concluding the proof of the lemma.
 \end{proof}

 \begin{proposition}\label{prop:3}
    Let us keep the same conventions and notations as in Section~\ref{sec:outlines}. Let $p\geq 1$. Assume that
     $$C_p\coloneq \int_{S_G}{\int_{X_H}{|\alpha(s^{-1},x)|^p_H\ d\nu_H(x)}\ d\la_G(s)}$$
     is finite, then we have
     $$\int_{X_G}{\|\gint_G f_x\|_p^p\ d\nu_G(x)}\leq C_p\|\gsup_H f\|_p^p.$$
 \end{proposition}

 \begin{proof}
     By Lemmas~\ref{lem:1} and~\ref{lem:2}, we have
     $$\int_{X_G}{\|\gint_G f_x\|_p^p\ d\nu_G(x)}=\int_{S_G}{\int_{X_H}{\|\lambda(\alpha(s^{-1},x)^{-1}) f-f\|_p^p\ d\nu_H(x)}\ d\la_G(s)}.$$
     Lemma~\ref{lem:3} now implies
     $$||\lambda(\alpha(s^{-1},x)^{-1}) f-f||_p^p \leq |\alpha(s^{-1},x)|_H^p ||\gsup_H f||_p^p$$
     for every $s\in S_G$ and $x\in X_G$, and the result follows.
 \end{proof}

 \begin{proposition}\label{prop:4}
     Let us keep the same conventions and notations as in Section~\ref{sec:outlines}. Let $\varphi\colon\R_+\to\R_+$ be a non-decreasing map such that $t\mapsto t/\varphi(t)$ is non-decreasing. Assume that
     $$C_{\varphi}\coloneq \int_{S_G}{\int_{X_H}{\varphi(|\alpha(s^{-1},x)|_H\|\gsup_H f\|_1)\ d\nu_H(x)}\ d\la_G(s)}$$
     is finite, then we have
     $$\int_{X_G}{\|\gint_G f_x\|_1\ d\nu_G(x)}\leq C_{\varphi}\frac{2\|f\|_1}{\varphi(2\|f\|_1)}.$$
 \end{proposition}

 \begin{proof}
     Using again Lemmas~\ref{lem:1} and~\ref{lem:2}, we have
     \begin{align*}
         &\int_{X_G}{\|\gint_G f_x\|_1\ d\nu_G(x)}\\
         =&\int_{S_G}{\int_{X_H}{\|\lambda(\alpha(s^{-1},x)^{-1}) f-f\|_1\ d\nu_H(x)}\ d\la_G(s)}\\
         =&\int_{S_G}{\int_{X_H}{\frac{\|\lambda(\alpha(s^{-1},x)^{-1}) f-f\|_1}{\varphi(\|\lambda(\alpha(s^{-1},x)^{-1}) f-f\|_1)}\varphi(\|\lambda(\alpha(s^{-1},x)^{-1}) f-f\|_1)\ d\nu_H(x)}\ d\la_G(s)}.
     \end{align*}
     Let $(s,x)\in S_G\times X_H$. First, the monotonicity of $t\mapsto t/\varphi(t)$ and the inequality $\|\lambda(\alpha(s^{-1},x)^{-1}) f-f\|_1\leq 2\|f\|_1$ imply
     \begin{equation}\label{eq:1prop4}
         \frac{\|\lambda(\alpha(s^{-1},x)^{-1}) f-f\|_1}{\varphi(\|\lambda(\alpha(s^{-1},x)^{-1}) f-f\|_1)}\leq \frac{2\|f\|_1}{\varphi(2\|f\|_1)}.
     \end{equation}
     Secondly, the monotonicity of $\varphi$ and Lemma~\ref{lem:3} give
     \begin{equation}\label{eq:2prop4}
         \varphi(\|\lambda(\alpha(s^{-1},x)^{-1}) f-f\|_1)\leq \varphi(|\alpha(s^{-1},x)|_H\|\gsup_H f\|_1).
     \end{equation}
     We finally deduce the result from Inequalities~\eqref{eq:1prop4} and~\eqref{eq:2prop4}.
 \end{proof}

\begin{proof}[Proof of Theorem~\ref{th:Lp}]
    We assume without loss of generality that the groups are non-discrete (using Lemma~\ref{lem:nondiscrete}) and that the measure equivalence coupling satisfies the properties provided by Lemma~\ref{lem:HaarPreserving}. Let $v$ be a positive real number,  let $f\colon H\to\R$ whose support has measure less than $v$ and the associated maps $f_x$ as explained in Section~\ref{sec:outlines}. Let $K,K',C_p$ be the constants introduced in Propositions~\ref{prop:1},~\ref{prop:2} and~\ref{prop:3}. Note that $C_p$ is finite since the cocycle is $\ld^p$ and by the comments after Definition~\ref{def:IntCocycle} (namely $S_G$ can be replaced by $S_G^{-1}$).\par
    We claim that the set
    $$A\coloneq\left\{ x\in X_G\ \colon\|\gint_G f_x\|_p^p\leq \frac{(C_p+1)}{K'}\frac{\|\gsup_H f\|_p^p}{\|f\|_p^p}\|f_x\|_p^p\right\}$$
    has positive measure. Indeed, if it was not the case, then we would have
    $$\int_{X_G}{\|\gint_G f_x\|_p^p\ d\nu_G(x)}\geq \frac{(C_p+1)}{K'}\frac{\|\gsup_H f\|_p^p}{\|f\|_p^p}\int_{X_G}{\|f_x\|_p^p\ d\nu_G(x)},$$
    namely
    $$\int_{X_G}{\|\gint_G f_x\|_p^p\ d\nu_G(x)}\geq (C_p+1)\|\gsup_H f\|_p^p$$
    by Proposition~\ref{prop:2}, and this would contradict Proposition~\ref{prop:3}.\par
    Given $\varepsilon>0$, Markov's inequality and Proposition~\ref{prop:1} imply
    $$\nu_G\left (\left\{x\in X_G\ \colon\la_G(\supp f_x)\geq \frac{1}{\varepsilon}K\la_H(\supp f)\right\}\right )\leq\varepsilon,$$
    so assuming $\varepsilon<\mu(A)$, we find $x_0\in A$ such that $\la_G(\supp f_{x_0})\leq \frac{1}{\varepsilon}K\la_H(\supp f)\leq \frac{K}{\varepsilon}v$.\par
    This finally gives
    $$j^{\mathrm{int}}_{p,G}\left (\frac{K}{\varepsilon}v\right )\geq \frac{\|f_{x_0}\|_p}{\|\gint_Gf_{x_0}\|_p}\geq \left (\frac{K'}{C_p+1}\right )^{1/p}\frac{\|f\|_p}{\|\gsup_Hf\|_p}$$
    and taking the supremum over $f$, we get
    $$j^{\mathrm{int}}_{p,G}\left (\frac{K}{\varepsilon}v\right )\geq \left (\frac{K'}{C_p+1}\right )^{1/p}j^{\mathrm{sup}}_{p,H}(v),$$
    which completes the proof.
\end{proof}

\begin{proof}[Proof of Theorem~\ref{th:phi}]
    We assume without loss of generality that the groups are non-discrete (using Lemma~\ref{lem:nondiscrete}) and the measure equivalence coupling satisfies the properties provided by Lemma~\ref{lem:HaarPreserving}. Let $v$ be a positive real number,  let $f\colon H\to\R$ whose support has measure less than $v$ and the associated maps $f_x$ as explained in Section~\ref{sec:outlines}. We assume that $\|\gsup_H f\|_1=1$ (it is easy to see that the definition of the isoperimetric profile can be restricted to such functions). Let $K,K',C_{\varphi}$ be the constants introduced in Propositions~\ref{prop:1},~\ref{prop:2} and~\ref{prop:4}. Since we have $\|\gsup_H f\|_1=1$, the constant $C_{\varphi}$ does not depend on $f$. Moreover, this constant is finite by assumption and the comments after Definition~\ref{def:IntCocycle} (namely $S_G$ can be replaced by $S_G^{-1}$).\par
    We claim that the set
    $$A\coloneq\left\{ x\in X_G\ \colon\|\gint_G f_x\|_1\leq \frac{(C_{\varphi}+1)}{K'}\frac{2}{\varphi(2\|f\|_1)}\|f_x\|_1\right\}$$
    has positive measure. Indeed, if it was not the case, then we would have
    $$\int_{X_G}{\|\gint_G f_x\|_1\ d\nu_G(x)}\geq \frac{(C_{\varphi}+1)}{K'}\frac{2}{\varphi(2\|f\|_1)}\int_{X_G}{\|f_x\|_1\ d\nu_G(x)},$$
    namely
    $$\int_{X_G}{\|\gint_G f_x\|_1\ d\nu_G(x)}\geq (C_{\varphi}+1)\frac{2\|f\|_1}{\varphi(2\|f\|_1)}$$
    by Proposition~\ref{prop:2}, and this would contradict Proposition~\ref{prop:4}.\par
    Given $\varepsilon>0$, Markov's inequality and Proposition~\ref{prop:1} imply
    $$\nu_G\left (\left\{x\in X_G\ \colon\la_G(\supp f_x)\geq \frac{1}{\varepsilon}K\la_H(\supp f)\right\}\right )\leq\varepsilon,$$
    so assuming $\varepsilon<\mu(A)$, we find $x_0\in A$ such that $\la_G(\supp f_{x_0})\leq \frac{1}{\varepsilon}K\la_H(\supp f)\leq \frac{K}{\varepsilon}v$.\par
    This finally gives
    $$j^{\mathrm{int}}_{1,G}\left (\frac{K}{\varepsilon}v\right )\geq \frac{\|f_{x_0}\|_1}{\|\gint_Gf_{x_0}\|_1}\geq \frac{K'}{2(C_{\varphi}+1)}\varphi(2\|f\|_1)$$
    and taking the supremum over the maps $f$ satisfying $\|\gsup_Hf\|_1=1$, we get
    $$j^{\mathrm{int}}_{p,G}\left (\frac{K}{\varepsilon}v\right )\geq \frac{K'}{2(C_{\varphi}+1)}\varphi(2j_{1,H}(v))\geq \frac{K'}{2(C_{\varphi}+1)}\varphi(j^{\mathrm{sup}}_{1,H}(v)),$$
    which completes the proof.
\end{proof}

\section{Absence of quantitatively critical measure equivalence couplings}\label{sec:threshold}

Theorem~\ref{th:phi intro} states that, if there exists a $(\varphi,\ld^0)$-integrable measure equivalence coupling from $G$ to $H$, then most of the time we have the following bound on $\varphi$:
\begin{equation}\label{eq:upperbound}
    \varphi\preccurlyeq j_{1,G}\circ j_{1,H}^{-1},
\end{equation}
and we would like to know if this upper bound can be reached. For discrete amenable groups, the first-named author proved that such upper bound cannot be achieved~\cite[Theorem~B]{correiaAbsenceQuantitativelyCritical2025}, and we want to generalize it to the more general setting of locally compact compactly generated groups.\par
We need mild assumptions to get the inequality~\eqref{eq:upperbound}. First, we have to assume that $j_{1,H}$ is injective, so as to consider its inverse. We know that the isoperimetric profile is asymptotically equivalent to an injective map (see e.g.~\cite[Remark~1.2]{correiaAbsenceQuantitativelyCritical2025}) so, in some sense, we may assume without loss of generality that the isoperimetric profile is injective. Secondly, the inequality $\varphi\circ j_{1,H}\preccurlyeq j_{1,G}$ means that there exists a constant $C>0$ such that $\varphi\circ j_{1,H}(t)=O\left (j_{1,G}(Ct)\right )$, which means that $\varphi(t)=O\left (j_{1,G}(C j_{1,H}^{-1}(t))\right )$. To get rid of this constant and obtain the desired composition $j_{1,G}\circ j_{1,H}^{-1}$, we will assume that the isoperimetric profile of $G$ satisfies
\begin{equation*}
    \forall C>0,\ j_{1,G}(Ct)=O\left (j_{1,G}(t)\right ).
\end{equation*}
This condition already appeared in~\cite{erschlerIsoperimetricProfilesFinitely2003, correiaIsoperimetricProfilesLamplighterlike2025}. There is no known example of a compactly generated group whose isoperimetric does not satisfy this condition.\par
Extending the proof of~\cite[Theorem~B]{correiaAbsenceQuantitativelyCritical2025} to the case of locally compact compactly generated amenable groups, we want to show Theorem~\ref{th:ThresholdProfile intro} that we recall here.

\begin{theorem}\label{th:ThresholdProfile}
    Let $G$ and $H$ be unimodular locally compact compactly generated groups. Assume that there exist a non-decreasing function $f_G$ and an increasing function $f_H$ satisfying $f_G(n)\simeq j_{1,G}(n)$, $f_H(n)\simeq j_{1,H}(n)$ and the following assumptions as $t\to +\infty$:
		\begin{equation}\label{hyp1 prof}
			f_G(t)=o\left (f_H(t)\right ),
		\end{equation}
		\begin{equation}\label{hyp2 prof}
			\forall C>0,\ f_G(Ct)=O\left (f_G(t)\right ),
		\end{equation}
		\begin{equation}\label{hyp3 prof}
			\forall C>0,\ f_G\circ f_H^{-1}(Ct)=O\left (f_G\circ f_H^{-1}(t)\right ).
		\end{equation}
		Then there is no $(f_G\circ f_H^{-1},\ld^0)$-integrable measure equivalence coupling from $G$ to $H$.
\end{theorem}

About Theorem~\ref{th:volumegrowth intro} dealing with volume growth, we prove a similar statement, as in~\cite[Theorem~D]{correiaAbsenceQuantitativelyCritical2025}. This is Theorem~\ref{th:ThresholdGrowth intro} that we recall here.

\begin{theorem}\label{th:ThresholdGrowth}
    Let $G$ and $H$ be unimodular locally compact compactly generated groups. Assume that there exist two increasing functions $f_G$ and $f_H$ satisfying $f_G(n)\simeq V_G(n)$, $f_H(n)\simeq V_H(n)$ and the following assumptions as $t\to +\infty$:
		\begin{equation}\label{hyp1 growth}
			f_G^{-1}(t)=o\left (f_H^{-1}(t)\right ),
		\end{equation}
		\begin{equation}\label{hyp2 growth}
			\forall C>0,\ f_G^{-1}(Ct)=O\left (f_G^{-1}(t)\right ),
		\end{equation}
		\begin{equation}\label{hyp3 growth}
			\forall C>0,\ f_G^{-1}\circ f_H(Ct)=O\left (f_G^{-1}\circ f_H(t)\right ).
		\end{equation}
		Then there is no $(f_G^{-1}\circ f_H,\ld^0)$-integrable measure equivalence coupling from $G$ to $H$.
\end{theorem}

Theorems~\ref{th:ThresholdProfile} and~\ref{th:ThresholdGrowth} are direct consequences of the following crucial proposition.

\begin{proposition}\label{prop:ThresholdKeyProp}
    Let $G$ and $H$ be non-discrete unimodular locally compact compactly generated groups, let $G\act (X,\nu)$ be a free measure preserving action on a standard finite-measured space. Let $\varphi\colon\R_+\to\R_+$ be a sublinear and non-decreasing function. If an $H$-valued cocycle $\alpha\colon G\times X\to H$ is $\varphi$-integrable, then there exists a map $\psi\colon\R_+\to\R_+$ satisfying the following properties:
    \begin{itemize}
        \item $\varphi(t)\not\succcurlyeq\psi(t)$;
        \item the maps $\psi$ and $t\mapsto t/\psi(t)$ are non-decreasing;
        \item $\psi$ is subadditive;
        \item the cocycle $\alpha$ is $\psi$-integrable.
    \end{itemize}
\end{proposition}

We are now able to prove Theorems~\ref{th:ThresholdProfile} and~\ref{th:ThresholdGrowth}, using this proposition. The latter will be proved after.

\begin{proof}[Proof of Theorem~\ref{th:ThresholdProfile}]
By Lemma~\ref{lem:nondiscrete}, we can assume that the groups are non-discrete. For $\varphi\coloneq f_G\circ f_h^{-1}$, let us assume by contradiction that there exists a $(\varphi,\ld^0)$-integrable measure equivalence coupling from $G$ to $H$. Then Proposition~\ref{prop:ThresholdKeyProp} implies that this coupling is also $(\psi,\ld^0)$-integrable with a map $\psi\colon\R_+\to\R_+$ such that $\psi$ and $t\mapsto t/\psi(t)$ are non-decreasing, and $\psi\not\preccurlyeq f_G\circ f_H^{-1}$. But Theorem~\ref{th:ThresholdProfile} implies $\psi\circ j_{1,H}\preccurlyeq j_{1,G}$, namely $\psi\preccurlyeq f_G\circ f_H^{-1}$ (see the end of the proof of~\cite[Theorem~B]{correiaAbsenceQuantitativelyCritical2025} for a justification), so we get a contradiction.
\end{proof}

\begin{proof}[Proof of Theorem~\ref{th:ThresholdGrowth}]
We proceed as in the last proof. Indeed, the map $\psi\colon\R_+\to\R_+$ provided by Proposition~\ref{prop:ThresholdKeyProp} is also subadditive, an assumption of Theorem~\ref{th:ThresholdGrowth} which enables us to get a contradiction.
\end{proof}

We now proceed to the proof of Proposition~\ref{prop:ThresholdKeyProp}. To this end, we have the following key lemma. In the case of finitely generated groups, this has been proved at the beginning of the proof of~\cite[Lemma~3.2]{correiaAbsenceQuantitativelyCritical2025}, which crucially uses finiteness of the generating sets. Here we manage to improve the proof in the non-discrete case.

\begin{lemma}\label{lem:keylemmaThreshold}
    Let $G$ and $H$ be non-discrete unimodular locally compact compactly generated groups, and let $\varphi\colon\R_+\to\R_+$ be a map. Given a free measure preserving action $G\act (X,\nu)$ on a standard finite measured space, let $\alpha\colon G\times X\to H$ be a $\varphi$-integrable $H$-valued cocycle. Then there exist a measurable subset $A$ of $G$ and a sequence $(K_n)_{n\geq 1}$ of positive integers such that:
    \begin{itemize}
        \item $\la_G(A)>0$ and $A$ contains a countable dense subset of $G$;
        \item $K_n\underset{n\to +\infty}{\longrightarrow}+\infty$;
        \item for every $g\in A$, the series 
        \begin{equation*}
            \sum_{n\geq 1}{K_n\ \varphi(n)\ \nu(|c(g,\cdot)|_{H}=n)}
        \end{equation*}
        converges.
    \end{itemize}
\end{lemma}

\begin{proof}
    Let $D$ be a countable dense subset of $G$ and let us enumerate its elements: $D=\{g_1,g_2,g_3,\ldots\}$. Let $A_0$ be a measurable subset of $G$ such that $0<\la_G(A_0)<+\infty$. First note that we have
    \begin{equation*}
        \int_{X}{\varphi(|c(g,x)|_H)d\nu(x)}=\sum_{n\geq 0}{\varphi(n)\ \nu(|c(g,\cdot)|_{H}=n)}
    \end{equation*}
    for every $g\in G$, and these quantities converge by $\varphi$-integrability. We set
    \begin{equation*}
        N(g,j)\coloneq\min{\left\{N\geq 1\ \colon\sum_{n\geq N}{\varphi(n)\ \nu(|c(g,\cdot)|_{H}=n)}\leq\frac{1}{j^3}\right\}}
    \end{equation*}
    for every $g\in G$ and every integer $j\geq 1$. Then we have
    \begin{equation*}
        \la_G(\{g\in A_0\mid N(g,j)>N)\underset{N\to +\infty}{\longrightarrow} 0,
    \end{equation*}
    so that there exists an increasing sequence $(N_j)_{j\geq 1}$ of positive integers such that
    \begin{itemize}
        \item $N_j\underset{j\to +\infty}{\longrightarrow}+\infty$;
        \item the series $\sum_{j\geq 1}{\la_G(\{g\in A_0\mid N(g,j)>N_j)}$ converges;
        \item for every $j\geq 1$, for every $i\in\{1,\ldots,j\}$, $N(g_i,j)\leq N_j$.
    \end{itemize}
    First, given $i\geq 1$, we have
    $$\sum_{n\geq N_j}{\varphi(n)\ \nu(|c(g_i,\cdot)|_{H}=n)}\leq \frac{1}{j^3}$$
    for every $j\geq i$. Secondly, Borel-Cantelli lemma implies that for $\la_G$-almost every $g\in A_0$, let us say for every $g\in A_1$ with some $A_1\subset A_0$ satisfying $\la_G(A_0\setminus A_1)=0$, we have 
    $$\sum_{n\geq N_j}{\varphi(n)\ \nu(|c(g,\cdot)|_{H}=n)}\leq \frac{1}{j^3}$$
    for sufficiently large integers $j$.\par
    To conclude, we have proved that for every $g\in D\cup A_1$, we have
    \begin{equation*}
        \sum_{N_j\leq n\leq N_{j+1}-1}{j\ \varphi(n)\ \nu(|c(g,\cdot)|_{H}=n)}\leq \frac{1}{j^2}.
    \end{equation*}
    for every sufficiently large integer $j$. It now suffices to define $A=D\cup A_1$ and $K_n=j$ for every $n\in\{N_j,N_j+1,\ldots,N_{j+1}-1\}$ and we are done.
\end{proof}

\begin{proof}[Proof of Proposition~\ref{prop:ThresholdKeyProp}]
    Let $(K_n)_{n\geq 1}$ be the sequence provided by Lemma~\ref{lem:keylemmaThreshold}, together with $A\subset G$ of positive Haar measure as in its statement, and let us denote by $D$ a dense countable set contained in $A$. With the same method as in the proof of~\cite[Lemma~3.2]{correiaAbsenceQuantitativelyCritical2025}, we build a map $\psi\colon\R_+\to\R_+$ from this sequence, which satisfies the following properties:
    \begin{itemize}
        \item $\psi$ and $t\mapsto t/\psi(t)$ are non-decreasing;
        \item $\psi$ is subbadditive;
        \item $\varphi(t)\not\succcurlyeq\psi(t)$;
        \item for every $g\in A$, $\|\alpha(g,.)\|_{\psi}$ is finite.
    \end{itemize}
    Given a compact generating set $S_G$ of $G$, it now remains to prove that $\sup_{g\in S_G}{\|\alpha(g,.)\|_{\psi}}$ is finite. We use the same techniques as in~\cite[Appendix A.2 (before Lemma~A.1)]{baderIntegrableMeasureEquivalence2013}. For every $g\in G$, we abusively write
    $$\|g\|_{\psi}\coloneq\|\alpha(g,.)\|_{\psi}\in\R\cup\{+\infty\}.$$
    Note that the cocycle identity, the subadditivity of $\psi$ and the $G$-invariance of $\nu$ imply
    \begin{equation}\label{eq:subaddPsiNorm}
        \|g^{-1}\|_{\psi}=\|g\|_{\psi}\text{ and }\|gg'\|_{\psi}\leq\|g\|_{\psi}+\|g'\|_{\psi}
    \end{equation}
    for every $g,g'\in G$. Denoting
    \begin{equation*}
        E_t\coloneq\{g\in G\mid \|g\|_{\psi}<t\}
    \end{equation*}
    for every $t>0$, we have $\la_G\left (\bigcup_{t>0}{E_t}\right )\geq\la_G(A)>0$, so there exists $t_0>0$ such that $\la_G(E_{t_0})>0$. This implies that $E_{t_0}^{-1}\cdot E_{t_0}$ is a neighbourhood of $1_G$. But we also have $E_{t_0}^{-1}\cdot E_{t_0}=E_{t_0}\cdot E_{t_0}\subset E_{2t_0}$ using~\eqref{eq:subaddPsiNorm}, so $E_{2t_0}$ is a neighbourhood of $1_G$. By density of $D$ and compactness of $S_G$, there exists a finite subset $F$ of $D$ such that
    $$S_G\subset\bigcup_{g\in F}{gE_{2t_0}}.$$
    Once again by~\eqref{eq:subaddPsiNorm}, we obtain for every $s\in S_G$:
    $$\|s\|_{\psi}\leq 2t_0+\max_{g\in F}{\|g\|_{\psi}},$$
    and so $\sup_{s\in S_G}{\|s\|_{\psi}}<+\infty$. This completes the proof.
\end{proof}



\printbibliography

{\bigskip
		\footnotesize
		
		\noindent C.~Correia,
		\textsc{Université Paris Cité, Institut de Mathématiques de Jussieu-Paris Rive Gauche, 75013 Paris, France}\par\nopagebreak\noindent
		\textit{E-mail address: }\texttt{corentin.correia@imj-prg.fr}}

{\bigskip
		\footnotesize
			\noindent J.~Paucar, \textsc{Pontificia Universidad Católica del Perú, Departamento de Ciencias, Sección Matemáticas, Lima, Perú}\par\nopagebreak\noindent
		\textit{E-mail address: }\texttt{jlpaucar@pucp.edu.pe}}
		

\end{document}